\documentclass[a4paper,12pt]{article}
\usepackage{graphics}
\usepackage{psfrag}
\usepackage{latexsym}
\usepackage{graphicx}
\usepackage[bf, small]{titlesec}
\catcode`@=11
\renewcommand{\@begintheorem}[2]{\it \trivlist
      \item[\hskip \labelsep{\bf #1\ #2{\rm :}}]}
\renewcommand{\@opargbegintheorem}[3]{\it \trivlist
      \item[\hskip \labelsep{\bf #1\ #2\ {\rm (#3)\/:}}]}
\def\@sect#1#2#3#4#5#6[#7]#8{\ifnum #2>\c@secnumdepth
     \def\@svsec{}\else
     \refstepcounter{#1}\edef\@svsec{\csname the#1\endcsname{.}\hskip 1em }\fi
     \@tempskipa #5\relax
      \ifdim \@tempskipa>\z@
        \begingroup #6\relax
          \@hangfrom{\hskip #3\relax\@svsec}{\interlinepenalty \@M #8\par}
        \endgroup
       \csname #1mark\endcsname{#7}\addcontentsline
         {toc}{#1}{\ifnum #2>\c@secnumdepth \else
                      \protect\numberline{\csname the#1\endcsname}\fi
                    #7}\else
        \def\@svsechd{#6\hskip #3\@svsec #8\csname #1mark\endcsname
                      {#7}\addcontentsline
                           {toc}{#1}{\ifnum #2>\c@secnumdepth \else
                             \protect\numberline{\csname the#1\endcsname}\fi
                       #7}}\fi
     \@xsect{#5}}

\@addtoreset{equation}{section}
\newcommand{\Delete}[1]{}

\usepackage{amsmath,amssymb,amsthm}

\theoremstyle{plain}
\newtheorem{Thm}{Theorem}[section]
\newtheorem{Lem}[Thm]{Lemma}
\newtheorem{Prop}[Thm]{Proposition}
\newtheorem{Cor}[Thm]{Corollary}

\newtheorem{Rem}[Thm]{Remark}

\newtheorem*{Defi}{Definition}

\newcommand{\cC}{\ensuremath{\mathcal{C}}}

\newcommand{\cW}{\ensuremath{\mathcal{W}}}

\newcommand{\ck}{\star}

\title{The competition number of a graph in which
any two holes share at most one edge}

\author{
\begin{tabular}{c}
{\sc Jung Yeun LEE}\\
[1ex]
National Institute for Mathematical Sciences \\
Daejeon 305-390, Korea\\
\\
{\sc Suh-Ryung KIM}
\thanks{This research was supported by Basic Science Research Program 
through the National Research Foundation of Korea (NRF) funded 
by the Ministry of Education, Science and Technology (700-20100058).}\\
[1ex]
Department of Mathematics Education \\
Seoul National University, Seoul 151-742, Korea\\
\\
{\sc Yoshio SANO}
\thanks{
This work was supported by Priority Research Centers Program
through the National Research Foundation of Korea (NRF) funded by
the Ministry of Education, Science and Technology 
(MEST) (No. 2010-0029638).}
\thanks{Corresponding author. {\it e-mail address:} 
{\tt ysano@postech.ac.kr}}\\
[1ex]
Pohang Mathematics Institute \\
POSTECH, Pohang 790-784, Korea 
\end{tabular} }

\date{February 2011}

\begin{document}

\maketitle

\begin{abstract}
The competition graph of a digraph $D$ is a (simple undirected) graph
which has the same vertex set as $D$
and has an edge between $x$ and $y$ if and only if
there exists a vertex $v$ in $D$ such that
$(x,v)$ and $(y,v)$ are arcs of $D$.
For any graph $G$, $G$ together with sufficiently many isolated vertices
is the competition graph of some acyclic digraph.
The competition number $k(G)$ of $G$
is the smallest number of such isolated vertices.
In general, it is hard to compute the competition number $k(G)$
for a graph $G$
and it has been one of important research problems in the study of
competition graphs to characterize a graph by its competition number.

A hole of a graph is a cycle of length
at least $4$ as an induced subgraph.
It holds that the competition number of a graph
cannot exceed one plus the number of its holes
if $G$ satisfies a certain condition.
In this paper, we show that the competition number of
a graph with exactly $h$ holes any two of which share at most one edge
is at most $h+1$,
which generalizes the existing results on this subject.
\end{abstract}


\tableofcontents

\section{Introduction and Preliminaries}
\subsection{Introduction}

Let $D$ be an acyclic digraph.
The {\it competition graph} of $D$, denoted by $C(D)$,
is the (simple undirected) graph which
has the same vertex set as $D$
and has an edge between two distinct vertices $x$ and $y$ if and only if
there exists a vertex $v$ in $D$ such that $(x,v)$ and $(y,v)$
are arcs of $D$.
For any graph $G$, $G$ together with sufficiently many isolated vertices
is the competition graph of an acyclic digraph.
From this observation, Roberts \cite{cn}
defined the {\em competition number} $k(G)$ of
a graph $G$ to be the smallest number $k$ such that $G$ together with
$k$ isolated vertices is the competition graph of an acyclic digraph.

The notion of competition graph was introduced by Cohen~\cite{co} as a
means of determining the smallest dimension of ecological phase space.
Since then, various variations have been defined and studied by many
authors
(see \cite{kimsu,lu} for surveys).
Besides an application to ecology, the
concept of competition graph
can be applied to a variety of fields, as summarized
in \cite{RayRob}.
Roberts~\cite{cn} observed that characterization of competition
graph is equivalent to computation of competition number.
It does not seem to be easy in general to compute $k(G)$ for a given graph
$G$, as Opsut~\cite{op} showed that the computation of the
competition number of a graph is an NP-hard problem (see
\cite{kimsu,kr}  for graphs whose competition numbers are known).

It has been one of important research problems in the study of
competition graphs to characterize a graph by its competition number.
From this point of view,
we study the relationship between the competition number
and the number of holes of a graph. 

A cycle in a graph is called an {\it induced cycle} 
(also called a {\it chordless cycle} or a {\it simple cycle}) 
if it is an induced subgraph of the graph. 
A {\it hole} in a graph is 
an induced cycle of length at least $4$ in the graph. 
We denote the number of holes in a graph $G$ by $h(G)$. 
A graph without holes is called a {\em chordal graph}. 

The competition number of a graph with a few holes has been studied:

\begin{Thm}[Roberts \cite{cn}]\label{roberts}
Let $G$ be a chordal graph.
Then the competition  number of $G$ is at most $1$.
\end{Thm}

\begin{Thm}[Cho and Kim \cite{ck}]\label{comptwo}
Let $G$ be a graph with exactly one hole.
Then the competition number of $G$ is at most $2$.
\end{Thm}

\begin{Thm}[Lee, Kim, Kim, and Sano \cite{twoholes}, 
Li and Chang \cite{twoholes2}]\label{compthree}
Let $G$ be a graph with exactly two holes.
Then the competition number of $G$ is at most $3$.
\end{Thm}

Recently, it has been shown that 
the competition number of a graph with exactly $h$ holes 
is at most $h+1$ under several assumptions. 

\begin{Defi}[Li and Chang \cite{indep}]
{\rm
A hole $C$ of a graph $G$ is called {\it independent}
if, for any hole $C'$ of $G$,
\begin{itemize}
\item{}
$|V(C) \cap V(C')| \leq 2$.
\item{}
If $|V(C) \cap V(C')| = 2$,
then $|E(C) \cap E(C')|=1$ and $|V(C)| \geq 5$.
\qed
\end{itemize}
}
\end{Defi}

\begin{Thm}[\cite{indep}]\label{thm:indep}
Let $G$ be a graph with exactly $h$ holes
satisfying the following property {\rm (LC)}:
\begin{itemize}
\item[{\rm (LC)}]
All the holes of $G$ are independent.
\end{itemize}
Then the competition number of $G$ is at most $h+1$.
\end{Thm}

\begin{Thm}[Kamibeppu \cite{Kamibeppu}]
Let $G$ be a graph with exactly $h$ holes
satisfying the following property {\rm (K)}:
\begin{itemize}
\item[{\rm (K)}]
For any hole $C$ of $G$, there exists an edge $e_C$ of the hole $C$
such that the edge $e_C$ is not contained in any other induced cycle of $G$.
\end{itemize}
Then the competition number of $G$ is at most $h+1$.
\end{Thm}

\begin{Thm}[Kim, Lee, and Sano \cite{LKKS}]\label{thm:hed}
Let $G$ be a graph with exactly $h$ holes
satisfying the following property {\rm (E0)}:
\begin{itemize}
\item[{\rm (E0)}]
Any two distinct holes of $G$ are mutually edge disjoint.
\end{itemize}
Then the competition number of $G$ is at most $h+1$.
\end{Thm}

In this paper, we generalize
the above results except Theorem \ref{compthree}.
Our main result is the following: 

\begin{Thm}\label{thm:main}
Let $G$ be a graph with exactly $h$ holes
satisfying the following property:
\begin{itemize}
\item[{\rm (E1)}]
For any two distinct holes $C_1$ and $C_2$ of $G$,
$|E(C_1) \cap E(C_2)| \leq 1$.
\end{itemize}
Then the competition number of $G$ is at most $h+1$.
\end{Thm}

\subsection{Relationships among conditions}

We remark that
the class of graphs satisfying
the condition {\rm (E1)} is larger than
the class of graphs satisfying
one of the conditions (LC), (K), and (E0).
(See Figure \ref{fig:relationship}. 
Figure \ref{fig:ex} and Table \ref{tab:ex} 
give examples which shows that each region $i$ 
in Figure \ref{fig:relationship} 
is not empty.) 

\begin{figure}[h]
\begin{center}
\includegraphics[scale=0.9]{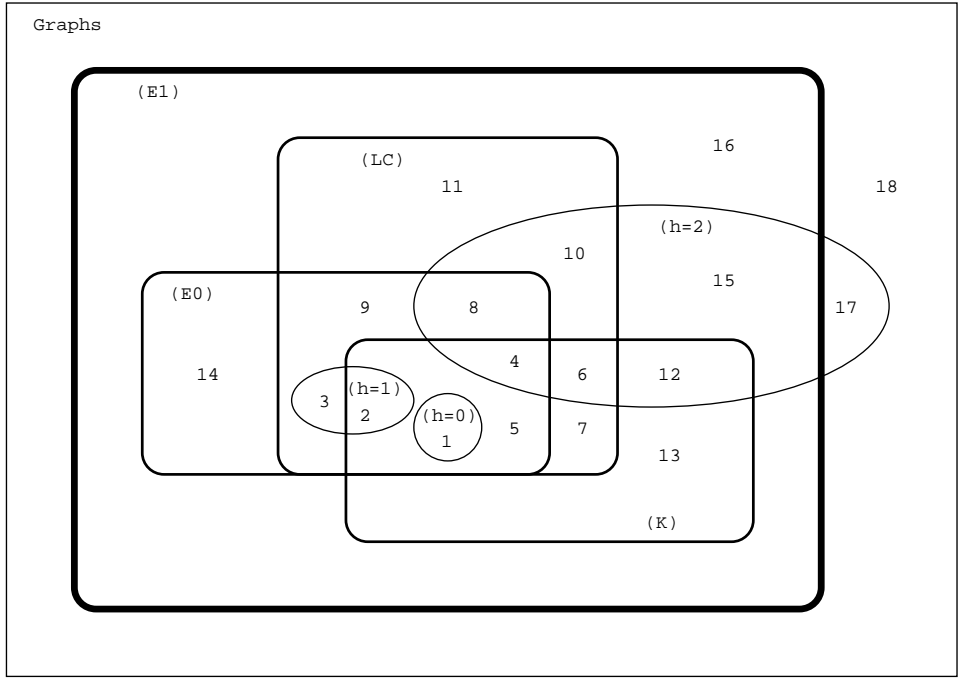}
\caption{Relationships among conditions, 
where (h=$i$) means the class of graphs having exactly $i$ holes 
($i=0,1,2$).}
\label{fig:relationship}
\end{center}
\end{figure}

\begin{Rem}
{\rm 
The disjointness of the class of graphs which satisfy the condition (E0) 
and do not satisfy the condition (LC) 
(the region 14 in Figure \ref{fig:relationship})
with the class of graphs satisfying the condition (K) 
or the class of graphs with $h(G)=2$ 
follows from \cite[Theorem 1.4]{LKKS}  
that states a graph which satisfies the condition (E0) 
and do not satisfy the condition (LC) must have an 
induced subgraph isomorphic to the complete tripartite graph 
$K_{2,2,2}$. 
Here $K_{2,2,2}$ has three holes which violate the condition (K).  
}
\end{Rem}

\begin{figure}[!h]
\begin{center}
\begin{tabular}{ccc}
\includegraphics[scale=0.7]
{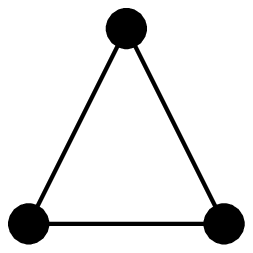} &
\includegraphics[scale=0.7]
{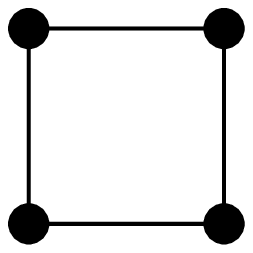} &
\includegraphics[scale=0.7]
{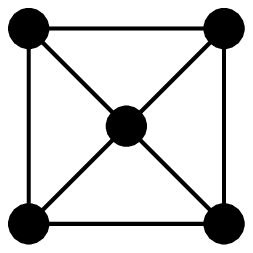}   \\ 
$\Gamma_1$ &
$\Gamma_2$ &
$\Gamma_3$  \\
& &  \\
\includegraphics[scale=0.7]
{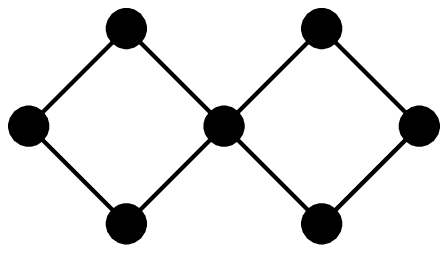} &
\includegraphics[scale=0.7]
{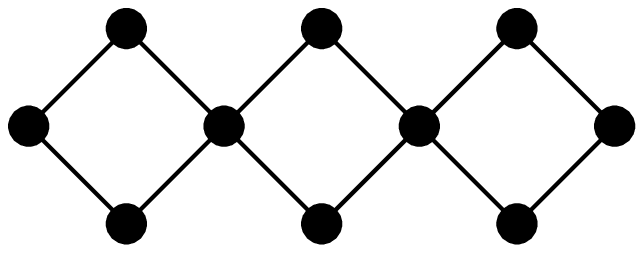} &
\includegraphics[scale=0.7]
{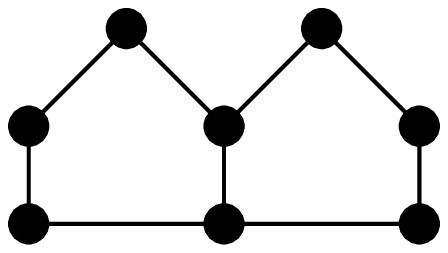}  \\ 
$\Gamma_4$ &
$\Gamma_5$ &
$\Gamma_6$ \\
& &  \\
\includegraphics[scale=0.7]
{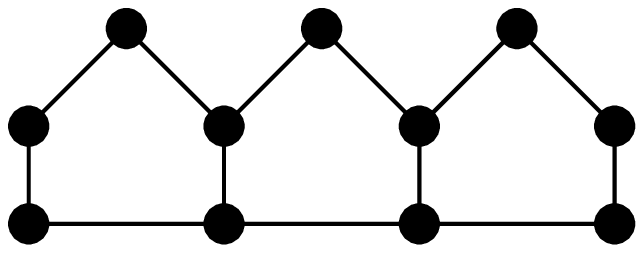} &
\includegraphics[scale=0.7]
{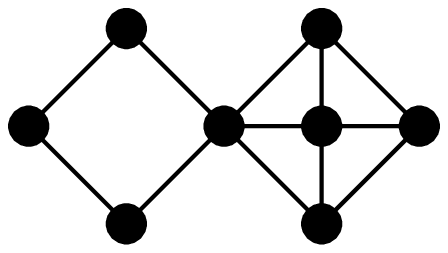} &
\includegraphics[scale=0.7]
{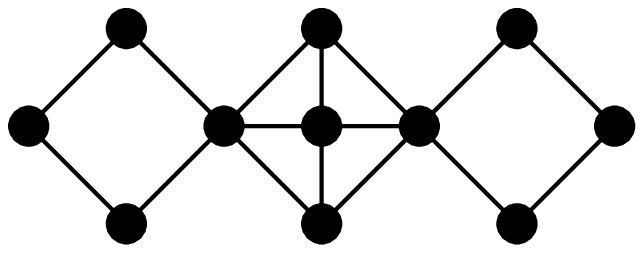}  \\ 
$\Gamma_7$ &
$\Gamma_8$ &
$\Gamma_9$ \\
& &  \\
\includegraphics[scale=0.7]
{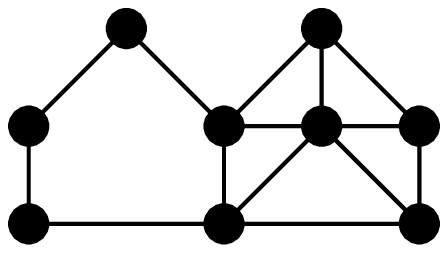} &
\includegraphics[scale=0.7]
{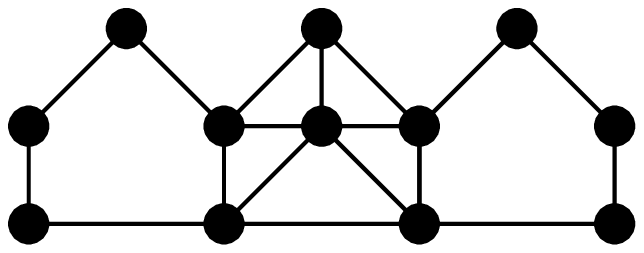} &
\includegraphics[scale=0.7]
{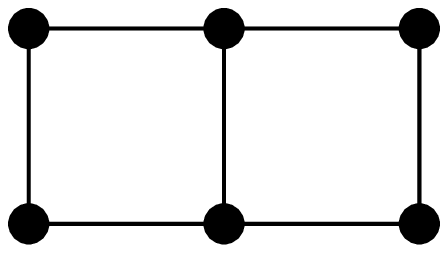} \\ 
$\Gamma_{10}$ &
$\Gamma_{11}$ &
$\Gamma_{12}$ \\
& & \\
\includegraphics[scale=0.7]
{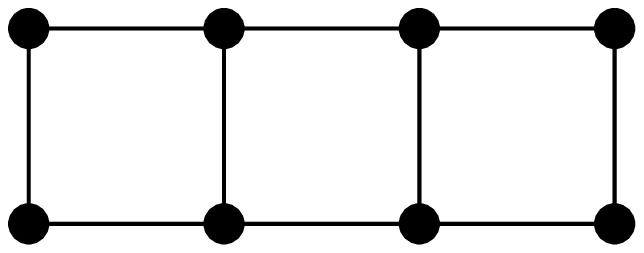} &
\includegraphics[scale=0.7]
{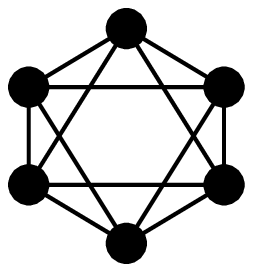} & 
\includegraphics[scale=0.7]
{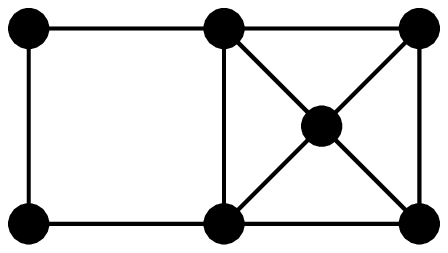} \\ 
$\Gamma_{13}$ &
$\Gamma_{14}$ &
$\Gamma_{15}$ \\
& &  \\
\includegraphics[scale=0.7]
{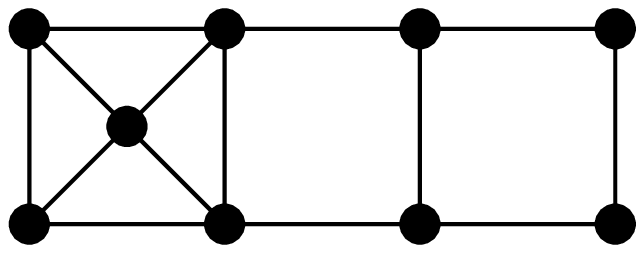} &
\includegraphics[scale=0.7]
{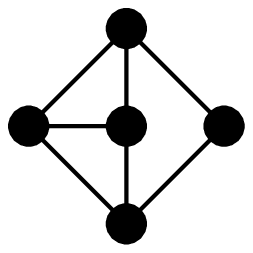} &
\includegraphics[scale=0.7]
{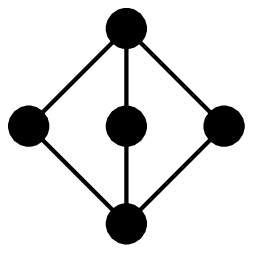} \\ 
$\Gamma_{16}$ &
$\Gamma_{17}$ &
$\Gamma_{18}$ \\
\end{tabular}
\end{center}
\caption{Examples of graphs where $\Gamma_i$ belongs to the region $i$ 
in Figure \ref{fig:relationship}}
\label{fig:ex}
\end{figure}

\begin{table}[h]
\begin{center}
\begin{tabular}{|c||c|c|c|c||c|}
\hline
Graph & $h(G)$ & (LC) & (K) & (E0) & (E1) \\
\hline
\hline

$\Gamma_{1}$ & 0 & $\ck$ & $\ck$ & $\ck$ & $\ck$  \\
\hline
$\Gamma_{2}$ & 1 & $\ck$ & $\ck$ & $\ck$ & $\ck$  \\
\hline
$\Gamma_{3}$ & 1 & $\ck$ & & $\ck$ & $\ck$  \\
\hline
$\Gamma_{4}$ & 2 & $\ck$ & $\ck$ & $\ck$ & $\ck$  \\
\hline
$\Gamma_{5}$ & 3 & $\ck$ & $\ck$ & $\ck$ & $\ck$  \\
\hline
$\Gamma_{6}$ & 2 & $\ck$ & $\ck$ &  & $\ck$   \\
\hline
$\Gamma_{7}$ & 3 & $\ck$ & $\ck$ &  & $\ck$  \\
\hline
$\Gamma_{8}$ & 2 & $\ck$ &  & $\ck$ & $\ck$  \\
\hline
$\Gamma_{9}$ & 3 & $\ck$ &  & $\ck$ & $\ck$  \\
\hline
$\Gamma_{10}$ & 2 & $\ck$ &  &  & $\ck$  \\
\hline
$\Gamma_{11}$ & 3 & $\ck$ &  &  & $\ck$ \\
\hline
$\Gamma_{12}$ & 2 &  & $\ck$ &  & $\ck$  \\
\hline
$\Gamma_{13}$ & 3 &  & $\ck$ &  & $\ck$  \\
\hline
$\Gamma_{14}$ & 3 &  &  & $\ck$ & $\ck$  \\
\hline
$\Gamma_{15}$ & 2 &  &  &  & $\ck$  \\
\hline
$\Gamma_{16}$ & 3 &  &  &  & $\ck$ \\
\hline
$\Gamma_{17}$ & 2 &  &  &  &  \\
\hline
$\Gamma_{18}$ & 3 &  &  &  &  \\
\hline
\end{tabular}
\end{center}
\caption{Graphs $\Gamma_i$ in Figure \ref{fig:ex} 
and the conditions (LC), (K), (E0), (E1)}
\label{tab:ex}
\end{table}

\begin{Prop}
If a graph $G$ satisfies the condition {\rm (LC)},
then $G$ also satisfies the condition {\rm (E1)}.
\end{Prop}

\begin{proof}
If a graph $G$ satisfies the condition (LC), then
it holds that $|V(C) \cap V(C')| \leq 2$
for any distinct holes $C$ and $C'$.
This implies that $|E(C) \cap E(C')| \leq 1$
holds for any distinct holes $C$ and $C'$.
Hence $G$ satisfies the condition {\rm (E1)}.
\end{proof}

The following proposition is rather obvious: 

\begin{Prop}
If a graph $G$ satisfies the hole-edge-disjoint condition
{\rm (E0)},
then $G$ also satisfies the condition {\rm (E1)}.
\end{Prop}

To show that 
the class of graphs satisfying 
the condition (K) is contained in the class of graphs 
satisfying the condition (E1), 
we prepare the following two lemmas: 

\begin{Lem}\label{lem:triorhole}
Let $G$ be a graph and $C$ be a cycle of $G$.
Then exactly one of the following holds: 
\begin{itemize}
\item[{\rm (a)}]
$C$ is an induced cycle of $G$, 
\item[{\rm (b)}]
There exist induced cycles $C_1, \ldots, C_s$ $(s \geq 2)$ in $G$ 
such that 
\begin{itemize}
\item[{\rm (1)}]
$V(C_i) \subsetneq V(C)$ $(i=1, \ldots, s)$, 
\item[{\rm (2)}]
Any edge $e$ of $C$ is an edge of 
$C_i$ for some $i \in \{1, \ldots, s\}$,  
\item[{\rm (3)}]
For any edge $e$ of $E(C_i) \setminus E(C)$, 
there exists unique $j\in \{1, \ldots, i-1,i+1, \ldots, s\}$
such that $e \in E(C_j)$. 
\end{itemize}
\end{itemize}
\end{Lem}

\begin{proof}
We show by induction on $|V(C)|$. 
If $|V(C)|=3$, then (a) holds and (b) does not holds
since any cycle of length $3$ is an induced cycle. 
Assume that the lemma holds for any cycle $C$ with 
$|V(C)|=t$. 
Let $C=v_0v_1 v_2 \cdots v_t v_0$ be a cycle with $|V(C)|=t+1$.  
Consider the subgraph $H$ of $G$ induced by $V(C)$. 
If $H$ is a cycle, then $C$ is an induced cycle in $G$ 
and so (a) holds. 
If $H$ is not a cycle, then $C$ is not an induced cycle in $G$ 
and so (a) does not hold. Now we show (b) holds. 
Note that any edge in $E(H) \setminus E(C)$ is a chord for $C$. 
Let $e^* = v_i v_j$ be a minimum chord for $C$, i.e., 
$|i-j|$ is smallest among all the chords for $C$. 
Then, the $(v_i,v_j)$-section $P_1$ of the cycle $C$ 
and the edge $e^*$ form an induced cycle $C^*$ in $G$ 
satisfying $V(C^*) \subsetneq V(C)$ and $e \in E(C^*)$ for $e \in E(P)$, 
and the $(v_j,v_i)$-section $P_2$ of the cycle $C$ and the edge $e^*$ 
form a cycle $C'$ in $G$ with $|V(C')| < t+1$. 
By the induction hypothesis, 
one of the following holds: 
(a)$'$ $C'$ is an induced cycle of $G$; 
(b)$'$ 
there exist induced cycles $C'_1, \ldots, C'_{s'}$ $(s' \geq 2)$ in $G$ 
such that the conditions (1)-(3) of (b) hold.  
If (a)$'$ holds, then let $\cC = \{C', C^*\}$.  
If (b)$'$ holds, then let $\cC = \{C'_1, \ldots, C'_{s'}, C^*\}$.  
In each case, the family $\cC$ of induced cycles in $G$ satisfies 
the conditions (1)-(3) of (b). Thus (b) holds. 
Hence the lemma holds. 
\end{proof}

The following lemma is well-known: 

\begin{Lem}\label{lem:sym-diff}
Let $C$ and $C'$ be two induced cycles in a graph $G$. 
Then, the subgraph of $G$ induced by 
the symmetric difference of $E(C)$ and $E(C')$ is
an edge-disjoint union of cycles of $G$. 
\end{Lem}

\begin{Prop}
If a graph $G$ satisfies the condition {\rm (K)},
then $G$ also satisfies the condition {\rm (E1)}.
\end{Prop}

\begin{proof}
Suppose that the condition {\rm (E1)} does not hold, 
i.e., there are two distinct holes $C$ and $C'$
such that $|E(C) \cap E(C')| \geq 2$. 
Consider the subgraph $H$ of $G$ 
induced by $(E(C) \cup E(C')) \setminus (E(C) \cap E(C'))$.  
By Lemma \ref{lem:sym-diff}, 
$H$ is an edge-disjoint union of cycles 
$C_1, \ldots, C_k$ $(k \geq 1)$ 
of $G$. 
Note that there is no triangle in $\{C_1, \ldots, C_k\}$ 
(for otherwise, an edge of a triangle would be a chord of 
the hole $C$ or the hole $C'$, which is a contradiction). 
If there is a hole $C_i$ in $\{C_1, \ldots, C_k\}$, 
then all the edges in $C_i$ are contained 
in the hole $C$ or the hole $C'$ 
and so the hole $C_i$ violates the condition (K). 
Therefore we may assume that any cycle in $\{C_1, \ldots, C_k\}$ 
is not an induced cycle. 
By Lemma \ref{lem:triorhole}, 
there exist induced cycles $C_{i,1}, \ldots, C_{i,s_i}$ $(s_i \geq 2)$ 
satisfying the conditions (1)-(3) of (b) in 
Lemma \ref{lem:triorhole} for each $C_i$ $(1 \leq i \leq k)$. 
Note that we can take $C_{i,j}$ so that 
every $C_{i,j}$ is different from 
the holes $C$ and $C'$ since $|E(C) \cap E(C')| \geq 2$. 
Let $\cC:=\{C,C'\} \cup \{C_{i,j} \mid i \in \{1, \ldots, k\}, 
j \in \{1, \ldots, s_i \} \}$. 
If there is a hole $C_{i,j}$ in the family $\cC$ other than $C$ 
and $C'$, then the hole $C_{i,j}$ violates the condition (K) 
since each edge in $C_{i,j}$ is contained 
in an induced cycle in 
$\{C,C', C_{i,1}, \ldots, C_{i,s_i} \}$. 
Therefore, 
all the induced cycles in $\cC$ other than $C$ and $C'$ 
should be triangles. 
However, then, each edge in $E(C) \setminus E(C')$ 
is contained in a triangle in $\cC \setminus \{C, C'\}$ 
and each edge in $E(C) \cap E(C')$ is contained in 
the hole $C'$. 
Thus the hole $C$ violates the condition (K). 
Hence the condition (K) does not hold in any case, 
and so the proposition holds. 
\end{proof}

\subsection{Preliminaries}

A set $S$ of vertices of a graph $G$ is called a {\it clique} of
$G$ if the subgraph of $G$ induced by $S$ is a complete graph.
A set $S$ of vertices of a graph $G$ is called a {\it vertex cut} of
$G$ if the number of connected components of $G-S$
is greater than that of $G$.

For a hole $C$ in a graph $G$, 
we denote by $X_C$ the set
of vertices which are adjacent to all the vertices of $C$:
\begin{equation}
X_C:=\{v \in V(G) \mid
uv \in E(G) \text{ {\rm for all }} u \in V(C) \}.
\end{equation}
Note that $V(C) \cap X_C = \emptyset$.
Given a walk $W$ of a graph $G$, we denote by $W^{-1}$ 
the walk represented by
the reverse of vertex sequence of $W$.
For a graph $G$ and a hole $C$ of $G$, we call a walk
(resp.\ path) $W$ a {\it $C$-avoiding walk}
(resp. {\it $C$-avoiding path})
if one of the following holds:
\begin{itemize}
\item{}
the length of $W$ is greater than or equal to $2$
and
none of the internal vertices of $W$ are in $V(C) \cup X_C$;
\item{}
the length of $W$ is $1$
and one of the two vertices of $W$ is not in $V(C) \cup X_C$.
\end{itemize}

Throughout this paper, we assume that all subscripts of vertices on a cycle 
are reduced to modular the length of the cycle.

\begin{Thm}[{\cite[Theorem 2.2]{twoholes}}]\label{chordalpart}
Let $G$ be a graph and $k$ be a nonnegative integer.
Suppose that $G$ has a subgraph $G_1$ with $k(G_1) \leq k$
and a chordal subgraph $G_2$ such that
$E(G_1) \cup E(G_2) =E(G)$
and $X:=V(G_1) \cap V(G_2)$ is a clique of $G_2$.
Then $k(G) \le k+1$.
\end{Thm}

\begin{Lem}[{\cite[Lemma 2.1]{LKKS}}]\label{wheel1}
Let $G$ be a graph and $C$ be a hole of $G$.
Let $x$ and $y$ be two non-adjacent vertices on $C$.
Suppose that there exists a common neighbor $v$ of $x$ and $y$ not on the hole $C$.
Then exactly one of the following holds:
\begin{itemize}
\item[{\rm (a)}]
$v \in X_C$;
\item[{\rm (b)}]
There exists a hole $C^*$ such that
$v \in V(C^*)$, $|E(C) \cap E(C^*)| \geq 2$,
and all the common edges are contained in
exactly one of the $(x,y)$-sections of $C$.
\end{itemize}
\end{Lem}

\begin{Lem}[{\cite[Lemma 2.4]{twoholes}}]\label{avoidingvertex}
Let $G$ be a graph and $C$ be a hole of $G$. 
Suppose that there exists a vertex $v$
such that 
$v$ is adjacent to consecutive vertices $v_i$ and $v_{i+1}$ of $C$,
and that 
$v$ is not in $X_C$ and not on any hole of $G$. 
Then, 
$v$ is not adjacent to any vertex in $V(C) \setminus \{v_i, v_{i+1}\}$.
\end{Lem}

\section{Structure of graphs satisfying the condition {\rm (E1)}}
\subsection{Properties of graphs satisfying the condition {\rm (E1)}}

\begin{Lem}\label{lem:K23-free}
Let $G$ be a graph satisfying the condition {\rm (E1)}.
Then $G$ is $K_{2,3}$-free.
\end{Lem}

\begin{proof}
Suppose that $G$ has an induced subgraph 
$H$ isomorphic to $K_{2,3}$ .
Let $V(H)=\{x_1, x_2, y_1, y_2, y_3\}$
and $E(H)=\{x_i y_j \mid i \in \{1,2\}, j \in \{1,2,3\}\}$.
Then $C_1 := x_1 y_1 x_2 y_2 x_1$ and
$C_2 := x_1 y_1 x_2 y_3 x_1$
are holes having two common edges $x_1y_1$ and $x_2y_1$,
which is a contradiction to the condition (E1).
\end{proof}

\begin{Prop}\label{prop:E1V2}
Let $G$ be a graph satisfying the condition {\rm (E1)}.
Then any two distinct holes of $G$ share at most two vertices.
\end{Prop}

\begin{proof}
By contradiction.
Suppose that there exist two distinct holes $C_1$ and $C_2$ in $G$
which share at least three vertices.
Let $u$, $v$ and $w$ be three distinct common vertices of $C_1$ and $C_2$.
Then they do not induce a triangle in $G$ since $C_1$ and $C_2$ are holes.
Without loss of generality, we may assume that $u$ and $v$ are not adjacent.
Let $P_1$ be the $(u,v)$-section of $C_1$ containing $w$
and let $P_2$ be the $(u,v)$-section of $C_2$ not containing $w$. 
(See Figure \ref{fig-2-2}.) 

Now we consider the subgraph $H$ of $G$ induced by $V(P_1) \cup V(P_2)$. 
Since $C_1 \neq C_2$, $P_1$ cannot be the other $(u,v)$-section of $C_2$ 
and $P_2$ cannot be the other section of $C_1$. 
Thus $H$ is distinct from $C_1$ and $C_2$. 
If $w$ is adjacent to an internal vertex in $P_2$,
then the edge is a chord of $C_2$ and we reach a contradiction.
Thus $w$ has degree $2$ in $H$. 
Since $w$ is an internal vertex of $P_1$, 
$w$ has its neighbors which are also on $P_1$. 
Let $a$ be a neighbor of $w$ closer to $u$ on $P_1$ 
and $b$ be the other neighbor of $w$. 
Then the $(a,u)$-section of $P_1$, $P_2$, and the $(v,b)$-section of $P_1$ 
form an $(a,b)$-walk in $H$ not containing $w$.
Let $P$ be a shortest $(a,b)$-path in $H$. 
Then, the edge $wa$, the $(a,b)$-path $P$, and the edge $bw$ 
form a cycle $C$. 
Since $H$ is an induced subgraph of $G$, 
$P$ is a shortest $(a,b)$-path in $G$.
Therefore the cycle $C$ is a hole in $G$.
Since $C$ is also a hole in $H$,
$C$ is distinct from the hole $C_1$. 
Now we reach a contradiction since 
the holes $C$ and $C_1$ share the two edges $wa$ and $wb$. 
\end{proof}

\begin{figure}[h]
\psfrag{a}{$a$}
\psfrag{b}{$b$}
\psfrag{u}{$u$}
\psfrag{v}{$v$}
\psfrag{w}{$w$}
\psfrag{P}{$P$}
\psfrag{C}{$C$}
\psfrag{H}{$H$}
\psfrag{P1}{$P_1$}
\psfrag{C1}{$C_{1}$}
\psfrag{P2}{$P_2$}
\psfrag{C2}{$C_{2}$}
\begin{center}
\includegraphics[scale=0.7]{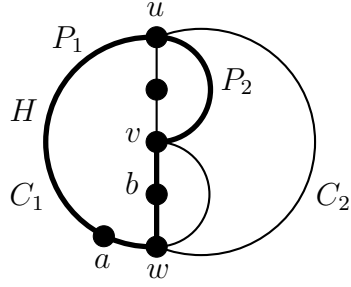}
\caption{A picture for Proof of Proposition \ref{prop:E1V2}}
\label{fig-2-2}
\end{center}
\end{figure}

\subsection{Properties of $X_C$ and $X_K$}

\begin{Lem}\label{wheel}
Let $G$ be a graph satisfying the condition {\rm (E1)}
and $C$ be a hole of $G$.
Let $x$ and $y$ be two non-adjacent vertices on $C$.
If there exists a common neighbor $v$ of $x$ and $y$ not on the hole $C$, 
then $v \in X_C$.
\end{Lem}

\begin{proof}
Since $G$ satisfies the condition {\rm (E1)},
Lemma~\ref{wheel1} (b) cannot happen
and thus the lemma holds.
\end{proof}

\begin{Lem}\label{lem:clique5}
Let $G$ be a graph satisfying the condition {\rm (E1)}
and $C$ be a hole of length at least $5$ in $G$.
Then $X_C$ is a clique.
\end{Lem}

\begin{proof}
By contradiction. Suppose that there are two non-adjacent vertices $x_1$ and $x_2$ in $X_C$.
Let $v_0 v_1 \cdots v_{m-1} v_0$ be the sequence of the vertices of the hole $C$ 
where $m \geq 5$.
Then $C_{(1)}:=x_1 v_0 x_2 v_2 x_1$ and
$C_{(2)}:= x_1 v_0 x_2 v_3 x_1$ are distinct holes of $G$
sharing the two edges $x_1 v_0$ and $x_2 v_0$,
which is a contradiction.
\end{proof}

We denote by $K_2^m$ a complete multipartite with $m$ parts
each of which has size $2$.
If $m=3$, then we denote $K_2^3$ also by $K_{2,2,2}$. 
We say that a graph is {\em $K_{2,2,2}$-free} if it does not
contain a complete tripartite graph $K_{2,2,2}$
as an induced subgraph.

\begin{Thm}\label{thm:strXC}
Let $G$ be a graph satisfying the condition {\rm (E1)}.
For any hole $C$ in $G$,
exactly one of the following holds:
\begin{itemize}
\item[{\rm (a)}]
$X_C$ is a clique.
\item[{\rm (b)}]
$C$ is contained in an induced subgraph of $G$ which is isomorphic to
$K_{2,2,2}$.
\end{itemize}
\end{Thm}

\begin{proof}
Suppose that (a) does not hold.
Then there are two non-adjacent vertices $x_1$ and $x_2$ in $X_C$.
By Lemma \ref{lem:clique5}, we have $|V(C)|=4$.
Therefore $V(C) \cup \{x_1,x_2\}$ induces $K_{2,2,2}$ and
thus (b) holds.
If (b) holds, then $|V(C)|=4$ and we can easily see 
that there are two non-adjacent vertices 
which are adjacent to all the vertices of $C$. Thus $(a)$ does not hold.
\end{proof}

\begin{Cor}\label{cor:clique4}
Let $G$ be a $K_{2,2,2}$-free graph satisfying the condition {\rm (E1)}
and $C$ be a hole in $G$.
Then $X_C$ is a clique.
\end{Cor}

\begin{proof}
It immediately follows from Theorem~\ref{thm:strXC}.
\end{proof}

For a vertex $v$ in a graph $G$, 
we denote by $N_G(v)$ the set of vertices adjacent to $v$ in $G$. 
We denote the set $N_G(v)\cup \{v\}$ by $N_G[v]$. 

For an induced subgraph $K$ of a graph $G$ 
isomorphic to $K_2^m$ for some $m \geq 2$, 
we denote by $X_K$ the set
of vertices which are adjacent to all the vertices of $K$:
\begin{equation}\label{eq:XK}
X_K:=\{v \in V(G) \mid
uv \in E(G) \text{ {\rm for all }} u \in V(K) \}.
\end{equation}

\begin{Lem}\label{multi}
Let $G$ be a graph satisfying the condition {\rm (E1)}.
Let $m$ be the largest integer
such that $G$ has an induced subgraph $K$ isomorphic to
$K_2^m$. 
If $m \geq 2$, then the following hold:
\begin{itemize}
\item[{\rm (1)}]
$X_K$ is a clique, 
\item[{\rm (2)}]
For two non-adjacent vertices $u$, $v$ in $K$, 
$N_G(u) \cap N_G(v) \subseteq X_K \cup V(K)$. 
\end{itemize}
\end{Lem}

\begin{proof}
We show (1) by contradiction. 
Suppose that there exist two nonadjacent vertices $x_1$ and $x_2$ in $X_K$. 
Then $V(K) \cup \{x_1,x_2\}$ induces a subgraph isomorphic to $K_2^{m+1}$, 
which contradicts the choice of $m$. 

Now we show (2). Let $u$ and $v$ be two non-adjacent vertices of $K$. 
If $N_G(u) \cap N_G(v) \subseteq V(K)$, 
then (2) holds and so we assume that 
$(N_G(u) \cap N_G(v)) \setminus V(K) \neq \emptyset$. 
Take a vertex $w \in (N_G(u) \cap N_G(v)) \setminus V(K)$. 
To show that $w\in X_K$, take any vertex $x$ of $K$. 
If $x \in \{u, v \}$, then $w$ is adjacent to $x$. 
Now we assume that $x \in V(K) \setminus \{u, v\}$. 
By the definition of $K$, 
$x$ is adjacent to both $u$ and $v$. 
Since $m \geq 2$, 
there exists a vertex $y$ of $K$ which is not adjacent to $x$. 
If $w$ is not adjacent to $x$, then $C_{(1)}:=uwvxu$ 
and $C_{(2)}:=uyvxu$ are two distinct holes 
sharing the two edges $ux$ and $vx$, 
which is a contradiction.
Thus the vertex $w$ is adjacent to $x$.
Since $x$ is chosen arbitrarily from $V(K)$, 
it holds that $w \in X_K$.
Hence we have
$(N_G(u) \cap N_G(v)) \setminus V(K) \subseteq X_K$, 
and thus 
$N_G(u) \cap N_G(v) \subseteq X_K \cup V(K)$.
\end{proof}

\subsection{Properties of $C$-avoiding paths
for a hole $C$ of length at least $5$}

\begin{Lem}\label{newedgedisjoint}
Let $G$ be a graph satisfying the condition {\rm (E1)}
and $C$ be a hole of length at least $5$ in $G$.
Then there is no $C$-avoiding path 
between two non-adjacent vertices of $C$.
\end{Lem}

\begin{proof} 
Let $C=v_0v_1\cdots v_{m-1}v_0$ be a hole of length at least $5$ in $G$, 
where $m \geq 5$. 
Suppose that there is a $C$-avoiding
$(v_i,v_j)$-path for some $i, j \in \{0,1,\ldots, m-1\}$
satisfying $|i-j| \ge 2$. 
Let $P$ be a shortest path among all the $C$-avoiding $(v_i,v_j)$-paths in $G$.
Then there is no edge joining two non-consecutive vertices on $P$.
Let $P_1$ and $P_2$ be the two $(v_i,v_j)$-sections of $C$
containing $v_{i-1}$ and $v_{i+1}$, respectively.
Then $P$ and $P_1$ form a cycle $C_{(1)}$
and $P$ and $P_2$ form a cycle $C_{(2)}$ in $G$.
Since both $C_{(1)}$ and $C_{(2)}$ share at least two edges with the hole $C$,
these cycles cannot be holes of $G$.
Since $C_{(1)}$ has a chord,
an internal vertex of $P$ is adjacent to an internal vertex on $P_1$.
Let $u$ be the first internal vertex on $P$
which is adjacent to an internal vertex on $P_1$.
Then let $v$ be the first internal vertex on $P_1$
which is adjacent to $u$. 
(See Figure \ref{fig-2-8}.) 
Then the $(v_i,u)$-section of $P$,
the edge $uv$, and the $(v,v_i)$-section of $P_1^{-1}$
form a triangle or a hole. 
In either case, it shares the edge $v_iv_{i-1}$ with $C$. 
Thus, by the condition (E1), $v=v_{i-1}$ and 
$u$ is the vertex immediately following $v_i$ on $P$. 
By applying a similar argument for $P_2$, we can show
that $u$ is adjacent to $v_{i+1}$.
Therefore, by Lemma~\ref{wheel},  we have $u \in X_C$.
However, since $P$ is a $C$-avoiding path,
$u$ does not belong to $X_C$
and thus we reach a contradiction.
\end{proof}

\begin{figure}[h]
\psfrag{vi}{$v_i$}
\psfrag{vj}{$v_j$}
\psfrag{v-}{$v_{i-1}$}
\psfrag{v+}{$v_{i+1}$}
\psfrag{u}{$u$}
\psfrag{v}{$v$}
\psfrag{P}{$P$}
\psfrag{C}{$C$}
\psfrag{P1}{$P_1$}
\psfrag{C1}{$C_{(1)}$}
\psfrag{P2}{$P_2$}
\psfrag{C2}{$C_{(2)}$}
\begin{center}
\includegraphics[scale=0.7]{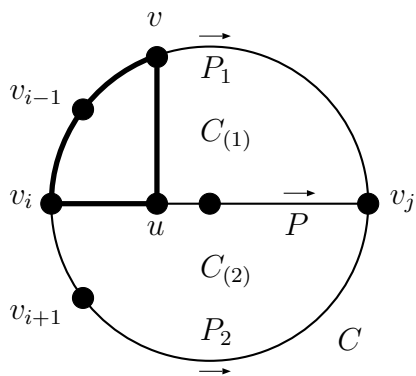}
\caption{A picture for Proof of Lemma \ref{newedgedisjoint}}
\label{fig-2-8}
\end{center}
\end{figure}

\begin{Cor}\label{newedgedisjoint1}
Let $G$ be a graph satisfying the condition {\rm (E1)}
and $C$ be a hole of length at least $5$ in $G$.
Given a vertex $v$ of $C$, adding new edges
joining $v$ and any other vertices on
$C$ reduces the number of holes.
\end{Cor}

\begin{proof}
It is obvious that $C$ is not a hole in the resulting graph $G'$. 
Thus it is sufficient to show that no new hole has been created. 
We show it by contradiction. 
Suppose that there is a hole $C'$ in $G'$ which is not in $G$. 
Then it contains an edge $vw$, where $w$ is a vertex on $C$ 
which is 
not adjacent to $v$ in $G$. 
Then 
$C'-vw$ 
is a $C$-avoiding $(v,w)$-path in $G$.
This contradicts Lemma~\ref{newedgedisjoint}.
\end{proof}

\begin{figure}[h]
\psfrag{v}{$v$}
\psfrag{w}{$w$}
\psfrag{C}{$C$}
\psfrag{C'}{$C'$}
\begin{center}
\includegraphics[scale=0.7]{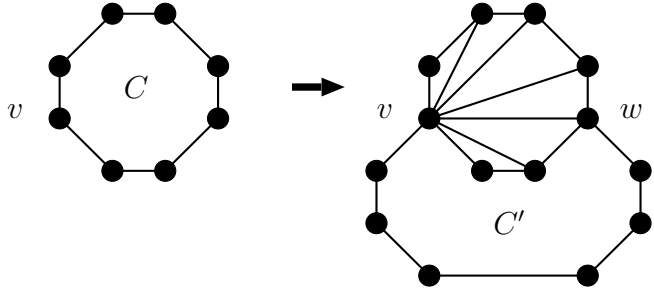}
\caption{A picture for Proof of Corollary \ref{newedgedisjoint1}}
\label{fig-2-9}
\end{center}
\end{figure}

\begin{Lem}[{\cite[Lemma 4]{ck}}]\label{vertexcutone}
Suppose that a graph $G$ has exactly one hole $C$.
If $G$ has a $C$-avoiding $(v_i,v_{i+1})$-path
for two adjacent vertices $v_i$ and $v_{i+1}$ on $C$, 
then $X_C \cup \{v_i,v_{i+1}\}$ is a vertex cut of $G$. 
\end{Lem}

We can extend this lemma as follows: 

\begin{Lem}\label{vertexcut}
Let $G$ be a graph satisfying the condition {\rm (E1)} 
and $C$ be a hole of length at least $5$ in $G$. 
If $G$ has a $C$-avoiding $(v_i,v_{i+1})$-path for 
two adjacent vertices $v_i$ and $v_{i+1}$ on $C$, 
then $X_C \cup \{v_i,v_{i+1}\}$ is a vertex cut of $G$. 
\end{Lem}

\begin{proof}
We prove by induction on the number $h$ of holes of a graph.
If a graph has exactly one hole, then it immediately follows 
from Lemma~\ref{vertexcutone}. 
Suppose that the lemma holds for any graph satisfying 
the condition {\rm (E1)} 
with at most $h-1$ holes for $h \ge 2$.
Now let $G$ be a graph satisfying the condition {\rm (E1)} with $h$ holes.
Suppose that $G$ has a $C$-avoiding $(v_i,v_{i+1})$-path
for some hole $C$ of $G$
and two adjacent vertices $v_i$ and $v_{i+1}$ on $C$.
Since $h \ge 2$,
there exists another hole $C'$.
Take a vertex $w$ of $C'$ and add new edges 
between $w$ and any other vertices on $C'$ by new edges.
Then, by Corollary~\ref{newedgedisjoint1},
the resulting graph $G'$ has less than $h$ holes.
Since no new hole has been created, $G'$ is still a
graph satisfying the condition {\rm (E1)}. By the condition {\rm (E1)}, 
$C$ and $C'$ share at most one edge and therefore
no chord for $C$ is created in the process of adding the edges. 
Thus $C$ is still a hole of $G'$.
By the induction hypothesis,
$X_C \cup \{v_i,v_{i+1}\}$ is a vertex cut of $G'$.
Since $G$ is a spanning subgraph of $G'$, it holds that
$X_C \cup \{v_i,v_{i+1}\}$ is a vertex cut of $G$.
\end{proof}

\subsection{Properties of $C$-avoiding paths for a hole $C$ of length $4$}

\begin{Prop}\label{prop:C4-avoid}
Let $G$ be a graph satisfying the condition {\rm (E1)}.
Suppose that $G$ has a hole $C=v_0 v_1 v_2 v_3 v_0$ of length $4$
and that there exists a $C$-avoiding $(v_0, v_{2})$-path 
of length at least $3$.
Let $P=x_0 x_1 x_2 \cdots x_{l-1} x_l$
be a shortest $C$-avoiding $(v_0, v_2)$-path, where $x_0=v_0$, $x_l=v_2$, 
and $l (\geq 3)$ is the length of $P$. 
Then, for any $i \in \{1, \ldots, l-1 \}$, the following hold:
\begin{itemize}
\item[{\rm (1)}]
$x_i$ is adjacent to exactly one of the vertices $v_1$, $v_3$;
\item[{\rm (2)}]
If $x_i v_1 \not\in E(G)$, then $x_{i+1} v_1 \in E(G)$;
\item[{\rm (3)}]
If $x_i v_3 \not\in E(G)$, then $x_{i+1} v_3 \in E(G)$.
\end{itemize}
\end{Prop}

\begin{proof}
We show (1) by contradiction.
Suppose that there is $i \in \{ 1, \ldots, l-1 \}$
such that $x_i$ is not adjacent to exactly one of vertices $v_1$, $v_3$.
First suppose that $x_i v_1 \in E(G)$ and $x_i v_3 \in E(G)$.
If $i \neq 1$, then
$v_0 v_1 x_i v_3 v_0$ is a hole and shares
two edges $v_0 v_1$ and $v_0 v_3$ with the hole $C$.
If $i=1$, then
$v_2 v_1 x_i v_3 v_2$ is a hole
and shares two edges $v_2 v_1$ and $v_2 v_3$ with $C$.

Suppose that $x_i v_1 \not\in E(G)$ and $x_i v_3 \not\in E(G)$.
Let $x_{i_1}$ (resp.\ $x_{i_3}$)
be the last vertex on the $(x_{0}, x_{i-1})$-section of $P$
that is adjacent to $v_1$ (resp.\ $v_3$),
and
let $x_{i_2}$ (resp.\ $x_{i_4}$)
be the first vertex on the $(x_{i+1}, x_{l})$-section of $P$
that is adjacent to $v_1$ (resp.\ $v_3$).
Then $C_{(1)}:= v_1 x_{i_1} x_{i_1+1} \cdots x_i \cdots x_{i_2} v_1$
and $C_{(2)}:= v_3 x_{i_3} x_{i_3+1} \cdots x_i \cdots x_{i_4} v_3$
are holes of $G$, and they share
two edges $x_{i-1} x_i$ and $x_i x_{i+1}$,
which is a contradiction. Hence (1) holds.

We show (2) by contradiction. Suppose that there is $i\in \{0, \ldots, l-1 \}$
such that $x_i v_1 \not\in E(G)$ and $x_{i+1} v_1 \not\in E(G)$.
Since $x_l=v_2$ and $v_1v_2 \in E(G)$, we have $i \neq l-1$.
By (1), $x_i v_3 \in E(G)$ and $x_{i+1} v_3 \in E(G)$.
Let
$x_{i_1}$ be the vertex defined in (1)
and
let $x_{i_5}$
be the first vertex on the $(x_{i_1+1}, x_{i})$-section of $P$
that is adjacent to $v_3$.
Then $C_{(3)}:= v_1 x_{i_1} x_{i_1+1} \cdots x_{i_5} v_1$
and
$C_{(4)} := v_1 x_{i_1} x_{i_1+1} \cdots x_{i_5} v_3 v_2 v_1$
are holes of $G$.
The two edges $v_1 x_{i_1}$ and $x_{i_1} x_{i_1 +1}$
are contained in both $C_{(3)}$ and $C_{(4)}$, which is a contradiction.
Hence it holds that
if $x_i v_1 \not\in E(G)$, then $x_{i+1} v_1 \in E(G)$.

Statement (3) can be shown by an argument similar to the proof of (2).
\end{proof}

We denote by $[x_1 y_1 | x_2 y_2 | x_3 y_3]$
the graph with vertex set
$\{x_1, x_2, x_3 \} \cup \{y_1, y_2, y_3 \}$
and edge set
$\{x_i x_j \mid 1\leq i<j \leq 3\}
\cup \{y_i y_j \mid 1\leq i<j \leq 3\}
\cup \{x_i y_i \mid 1\leq i \leq 3\}$.
A graph isomorphic to $[x_1 y_1 | x_2 y_2 | x_3 y_3]$
is called
a {\it $3$-prism graph}. 
In this paper,
we call a $3$-prism graph just a {\it prism}.
We say that a graph is {\it prism-free}
if the graph does not contain a prism
as an induced subgraph.

\begin{Prop}\label{prop:length3}
Let $G$ be a graph satisfying the condition {\rm (E1)}.
Suppose that $G$ has a hole $C=v_0 v_1 v_2 v_3 v_0$ of length $4$,
and that there is a $C$-avoiding $(v_0, v_2)$-path.
Let $P$ be a shortest $C$-avoiding $(v_0, v_2)$-path. 
Then the length of $P$
is equal to $3$ and the subgraph of $G$ induced by $V(C) \cup V(P)$ 
is a prism $[v_0 v_3 | x y | v_1 v_2]$ or a prism
$[v_0 v_1 | x y | v_3 v_2]$, where $P=v_0 xy v_2$. 
\end{Prop}

\begin{proof}
Let
$P=x_0 x_1 x_2 \cdots x_{l-1} x_l$ 
be a shortest $C$-avoiding $(v_0, v_2)$-path,
where $x_0=v_0$ and $x_l=v_2$.
Since $v_0v_2 \not\in E(G)$, $l \neq 1$.
If $l=2$, then $P=v_0 x_1 v_2$
and so
we have $x_1 \in X_C$ by Lemma \ref{wheel},
which contradicts the fact that
$P$ is a $C$-avoiding path.
Thus the length $l$ of $P$ is at least $3$. Suppose that $l \ge 4$. 
Then $x_3 \neq v_2$.
By Proposition \ref{prop:C4-avoid} (1), 
exactly one of
$x_1 v_1$, $x_1 v_3$ is an edge of $G$.
Without loss of generality,
we may assume that $x_1 v_1 \in E(G)$
and $x_1 v_3 \not\in E(G)$.
Then, by Proposition \ref{prop:C4-avoid} (3), $x_2v_3 \in E(G)$.
By (1) of the same proposition, $x_2v_1 \not\in E(G)$. 
By (2), $x_3v_1 \in E(G)$. Then, by (1), $x_3v_3 \not\in E(G)$. 
(See Figure \ref{fig-2-13}.) 
Then $C_{(1)}:= v_0 v_1 x_3 x_2 v_3 v_0$
and $C_{(2)}:= v_1 x_3 x_2 x_1 v_1$ are holes of $G$.
The two edges $v_1 x_3$ and $x_2 x_3$
are contained in
both $C_{(1)}$ and $C_{(2)}$, which is a contradiction 
to the fact that $G$ satisfy the condition (E1).
Hence $l=3$.
Furthermore, by Proposition \ref{prop:C4-avoid}, 
$Y$ 
the subgraph of $G$ induced by $V(C) \cup V(P)$ 
is either a prism $[v_0 v_3 | x y | v_1 v_2]$ or a prism 
$[v_0 v_1 | x y | v_3 v_2]$.
\end{proof}

\begin{figure}[h]
\psfrag{a}{$v_0$}
\psfrag{b}{$v_1$}
\psfrag{c}{$v_2$}
\psfrag{d}{$v_3$}
\psfrag{x1}{$x_1$}
\psfrag{x2}{$x_2$}
\psfrag{x3}{$x_3$}
\begin{center}
\includegraphics[scale=0.7]{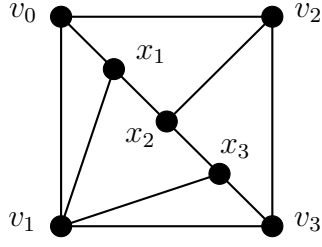}
\caption{A picture for Proof of Proposition \ref{prop:length3}}
\label{fig-2-13}
\end{center}
\end{figure}

Let $G$ a graph satisfying the condition {\rm (E1)} and 
$C=v_0 v_1 v_2 v_3 v_0$ be a hole of length $4$ in $G$. 
For two distinct prisms $Y_1$ and $Y_2$ containing $C$, 
we say that $Y_1$ and $Y_2$ are {\it of the same type} 
if a triangle in $Y_1$ has a common edge with one of the two triangles in $Y_2$,  
and 
we say that $Y_1$ and $Y_2$ are {\it of different types} 
if both of the two triangles in $Y_1$ have no common edge with 
the two triangles in $Y_2$. 
That is, 
two prisms of the forms $[v_0 v_3 | x y | v_1 v_2]$ and 
$[v_0 v_3 | x' y' | v_1 v_2]$ are of the same type, 
two prisms of the forms 
$[v_0 v_1 | x y | v_3 v_2]$ and 
$[v_0 v_1 | x' y' | v_3 v_2]$ are of the same type, 
and 
two prisms of the forms 
$[v_0 v_1 | x y | v_3 v_2]$ and 
$[v_0 v_3 | x' y' | v_1 v_2]$ are of different types. 
(See Figure \ref{fig:prism}.)

\begin{figure}[h]
\psfrag{a}{$v_0$}
\psfrag{b}{$v_1$}
\psfrag{c}{$v_2$}
\psfrag{d}{$v_3$}
\psfrag{e}{$x$}
\psfrag{f}{$y$}
\psfrag{g}{$v_0$}
\psfrag{h}{$v_1$}
\psfrag{i}{$v_2$}
\psfrag{j}{$v_3$}
\psfrag{k}{$x$}
\psfrag{l}{$y$}
\begin{center}
\includegraphics[scale=0.7]{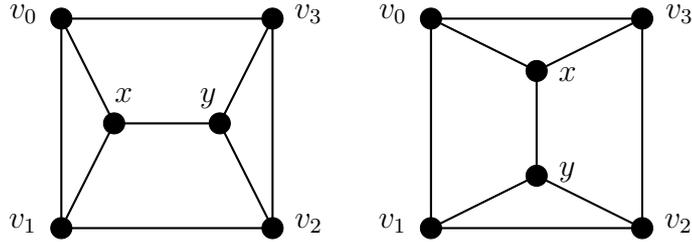}
\caption{Prisms $[v_0 v_3|x y|v_1 v_2]$ 
and $[v_0 v_1|x y|v_3 v_2]$}
\label{fig:prism}
\end{center}
\end{figure}

\begin{Cor}\label{cor:02-13}
Let $G$ be a graph satisfying the condition {\rm (E1)}.
Suppose that $G$ has a hole $C=v_0 v_1 v_2 v_3 v_0$ of length $4$. 
Then, there is a  $C$-avoiding $(v_0, v_2)$-path if and only if 
there is a $C$-avoiding $(v_1,v_3)$-path. 
\end{Cor}

\begin{proof}
Suppose that there is a $C$-avoiding $(v_0, v_2)$-path.
Let $P$ be a shortest path among all $C$-avoiding $(v_0, v_2)$-paths.
By Proposition \ref{prop:length3}, the length of $P$ is equal to $3$.
Let $P=v_0xyv_2$.
By Proposition \ref{prop:length3},
either $v_1 xy v_3$ or $v_1yxv_3$ is a $C$-avoiding $(v_1, v_3)$-path. 
(See Figure \ref{fig:prism}.)
We can show the converse similarly.
\end{proof}

\begin{Lem}\label{lem:sametype}
Let $G$ be a graph satisfying the condition {\rm (E1)}.
Suppose that $G$ has a hole $C$ of length $4$.
Then the prisms containing $C$ must be 
of the same type.
\end{Lem}

\begin{proof}
Let $C :=v_0 v_1 v_2 v_3 v_0$.
Suppose that
$C$ is contained in
prisms $Y_1$ and $Y_2$ of different types.
Without loss of generality, we may assume that
$Y_1 = [v_0 v_1 | x_1 x_2 | v_3 v_2]$ and
$Y_2 = [v_0 v_3 | y_1 y_2 | v_1 v_2]$
for some $x_1$, $x_2$, $y_1$, $y_2 \in V(G)$. 
(See Figure \ref{fig-2-15}.) 
Suppose that one of $x_1$, $x_2$ and one of $y_1$, $y_2$ are
adjacent.
By the symmetry,
we may assume that $x_1$ and $y_1$ are adjacent.
Then $C_{(1)}:=v_1 v_2 v_3 x_1 y_1 v_1$ is a hole of $G$.
But the edges $v_1 v_2$ and $v_2 v_3$ are contained in both
$C$ and $C_{(1)}$, which is a contradiction to the condition (E1).
Therefore, there is no edge between $\{x_1, x_2\}$ and $\{y_1, y_2\}$.
Then $C_{(2)}:= v_0 x_1 x_2 v_2 y_2 y_1 v_0$ and
$C_{(3)}:= v_0 x_1 x_2 v_1 v_0$ are holes of $G$ and
they share the two edges $v_0 x_1$ and $x_1 x_2$,
which is a contradiction to the condition (E1). 
Hence, 
the prisms containing $C$ must be 
of the same type.
\end{proof}

\begin{figure}[h]
\psfrag{a}{$v_0$}
\psfrag{b}{$v_1$}
\psfrag{c}{$v_2$}
\psfrag{d}{$v_3$}
\psfrag{x1}{$x_1$}
\psfrag{x2}{$x_2$}
\psfrag{y1}{$y_1$}
\psfrag{y2}{$y_2$}
\begin{center}
\includegraphics[scale=0.7]{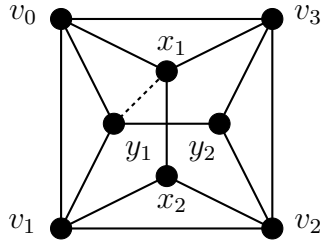}
\caption{A picture for Proof of Lemma \ref{lem:sametype}}
\label{fig-2-15}
\end{center}
\end{figure}

Let $K_t \Box K_2$ be the graph defined by
$V(K_t \Box K_2)=\{x_1, \ldots, x_t \} \cup \{y_1, \ldots, y_t \}$ 
and $E(K_t \Box K_2)=\{x_i x_j \mid 1\leq i<j \leq t\} 
\cup \{y_i y_j \mid 1\leq i<j \leq t\} 
\cup \{x_i y_i \mid 1\leq i \leq t\}$. 
(See Figure \ref{fig:K3-K2_K4-K2} for $t=3,4$.) 
Note that $K_3 \Box K_2$ is a prism.

\begin{figure}[h]
\psfrag{x13}{$x_1$}
\psfrag{x23}{$x_2$}
\psfrag{x33}{$x_3$}
\psfrag{y13}{$y_1$}
\psfrag{y23}{$y_2$}
\psfrag{y33}{$y_3$}
\psfrag{x14}{$x_1$}
\psfrag{x24}{$x_2$}
\psfrag{x34}{$x_3$}
\psfrag{x44}{$x_4$}
\psfrag{y14}{$y_1$}
\psfrag{y24}{$y_2$}
\psfrag{y34}{$y_3$}
\psfrag{y44}{$y_4$}
\begin{center}
\includegraphics[scale=0.7]{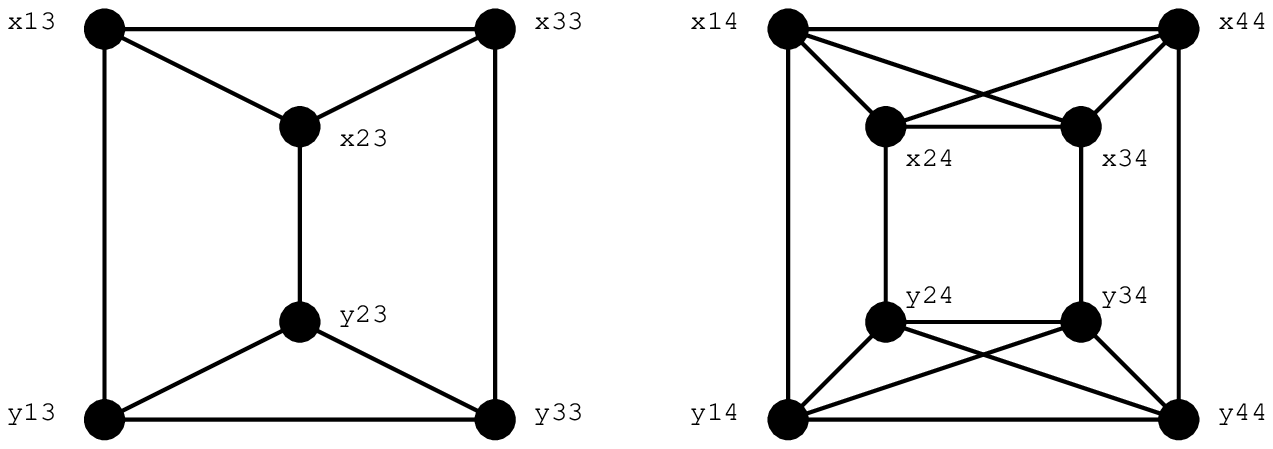}
\caption{$K_3 \Box K_2$ and $K_4 \Box K_2$}
\label{fig:K3-K2_K4-K2}
\end{center}
\end{figure}

\begin{Lem}\label{lem:twopaths}
Let $G$ be a graph satisfying the condition {\rm (E1)} and
$C:=v_0 v_1 v_2 v_3 v_0$ be a hole of length $4$ in $G$. 
Suppose that there exist two distinct $C$-avoiding $(v_0,v_2)$-paths.  
Let $P_1$ and $P_2$ be two distinct shortest $(v_0,v_2)$-paths.
Then the subgraph of $G$ induced by $V(C) \cup V(P_1) \cup V(P_2)$ 
is $K_4 \Box K_2$. 
\end{Lem}

\begin{proof}
By Lemma~\ref{prop:length3}, 
the lengths of $P_1$ and $P_2$ are equal to $3$. 
Let $P_1:=v_0x_1x_2v_2$ and $P_2:=v_0y_1y_2v_2$. 
By Lemma~\ref{prop:length3}, 
$V(C)\cup V(P_1)$ and  $V(C)\cup V(P_2)$ induce two distinct prisms 
$Y_1$ and $Y_2$. 
In addition, by Lemma \ref{lem:sametype}, 
$Y_1$ and $Y_2$ are of the same type. 
Without loss of generality, 
we may assume that 
$Y_1=[v_0 v_1 | x_1 x_2 | v_3 v_2]$ and 
$Y_2=[v_0 v_1 | y_1 y_2 | v_3 v_2]$ 
for some $x_1$, $x_2$, $y_1$, $y_2 \in V(G)$. 

First we show $x_1 y_2, x_2 y_1 \not\in E(G)$.
Suppose that $x_1 y_2 \in E(G)$ or
$x_2 y_1 \in E(G)$.
Without loss of generality,
we may assume that $x_1 y_2 \in E(G)$. 
(See Figure \ref{fig-2-16} (a).)
Then $C_{(1)}:=v_0 x_1 x_2 v_1 v_0$
and $C_{(2)}:=v_0 x_1 y_2 v_1 v_0$
are holes of $G$.
The two edges $v_0 x_1$ and $v_0 v_1$
are contained in both $C_{(1)}$ and $C_{(2)}$,
which is a contradiction to the condition (E1). 
Thus $x_1 y_2, x_2 y_1 \not\in E(G)$.
Second we show $x_1y_1 \in E(G)$. 
Suppose that $x_1$ and $y_1$ are not adjacent.
If  $x_2y_2 \in E(G)$, then let
$C_{(3)}:= v_0 x_1 x_2 y_2 y_1 v_0$
and
$C_{(4)}:= v_3 x_1 x_2 y_2 y_1 v_3$. 
(See Figure \ref{fig-2-16} (b).)
If $x_2y_2 \not\in E(G)$, then let
$C_{(3)}:= v_0 x_1 x_2 v_2 y_2 y_1 v_0$
and
$C_{(4)}:= v_3 x_1 x_2 v_2 y_2 y_1 v_3$. 
(See Figure \ref{fig-2-16} (c).)
Then, in both cases, $C_{(3)}$ and $C_{(4)}$ are holes in $G$.
Moreover,
the $(x_1, y_1)$-section $P$
of $C_{(3)}$ not containing $v_0$
coincides with the $(x_1, y_1)$-section
of $C_{(4)}$ not containing $v_3$.
Since $P$ contains at least $2$ edges, we reach a contradiction.
Thus $x_1$ and $y_1$ are adjacent.
The same argument holds for $x_2$ and $y_2$, 
and so it follows that $x_2$ and $y_2$ are adjacent.
Hence, 
the subgraph of $G$ induced by 
$V(C) \cup V(P_1) \cup V(P_2)$ 
is $K_4 \Box K_2$.
\end{proof}

\begin{figure}[h]
\psfrag{aa}{$v_0$}
\psfrag{ba}{$v_1$}
\psfrag{ca}{$v_2$}
\psfrag{da}{$v_3$}
\psfrag{x1a}{$x_1$}
\psfrag{x2a}{$x_2$}
\psfrag{y1a}{$y_1$}
\psfrag{y2a}{$y_2$}
\psfrag{ab}{$v_0$}
\psfrag{bb}{$v_1$}
\psfrag{cb}{$v_2$}
\psfrag{db}{$v_3$}
\psfrag{x1b}{$x_1$}
\psfrag{x2b}{$x_2$}
\psfrag{y1b}{$y_1$}
\psfrag{y2b}{$y_2$}
\psfrag{ac}{$v_0$}
\psfrag{bc}{$v_1$}
\psfrag{cc}{$v_2$}
\psfrag{dc}{$v_3$}
\psfrag{x1c}{$x_1$}
\psfrag{x2c}{$x_2$}
\psfrag{y1c}{$y_1$}
\psfrag{y2c}{$y_2$}
\psfrag{(a)}{(a)}
\psfrag{(b)}{(b)}
\psfrag{(c)}{(c)}
\begin{center}
\includegraphics[scale=0.7]{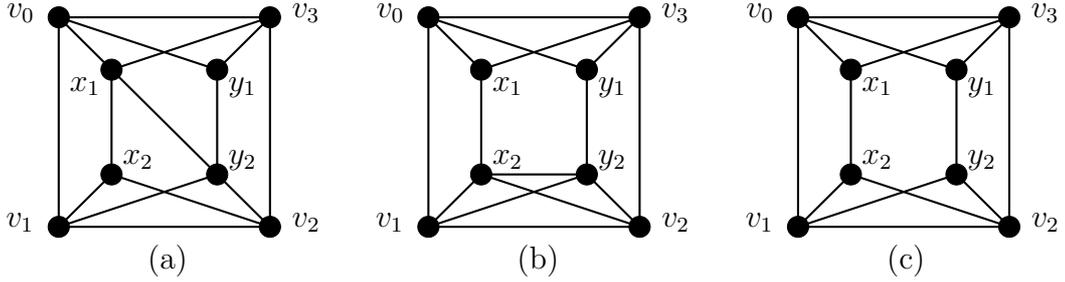}
\caption{Pictures for Proof of Lemma \ref{lem:twopaths}}
\label{fig-2-16}
\end{center}
\end{figure}

\begin{Thm}\label{thm:str}
Let $G$ be a graph satisfying the condition {\rm (E1)}.
For any hole $C:=v_0 v_1 v_2 v_3 v_0$ of length $4$ in $G$,
exactly one of the following holds.
\begin{itemize}
\item[{\rm (a)}]
$G$ has no $C$-avoiding $(v_0, v_2)$-path and
no $C$-avoiding $(v_1, v_3)$-path.
\item[{\rm (b)}]
$C$ is contained in an induced subgraph of $G$
which is isomorphic to $K_t \Box K_2$ for some $t \geq 3$.
\end{itemize}
\end{Thm}

\begin{proof}
Suppose that (a) does not hold.
By Corollary \ref{cor:02-13},
$G$ has a $C$-avoiding $(v_0, v_2)$-path.
Let $P_1, \ldots, P_s$ $(s \geq 1)$
be the shortest $C$-avoiding $(v_0, v_2)$-paths. 
If $s=1$, then 
$V(C) \cup V(P_1)$ induces a prism $K_3 \Box K_2$ 
by Proposition~\ref{prop:length3}. 
If $s \geq 2$, 
then, by Lemma~\ref{lem:twopaths}, $V(C) \cup V(P_i) \cup V(P_{j})$ 
induces $K_4 \Box K_2$ for any $1 \leq i < j \leq s$. 
Thus the subgraph 
of $G$
induced by $V(C) \cup V(P_1) \cup \cdots \cup V(P_{s})$ is isomorphic 
to $K_{s+2} \Box K_2$ and contains the hole $C$.
Hence (b) holds.

Suppose that (b) holds. 
Let $H$ be an induced subgraph of $G$ isomorphic to 
$K_{t} \Box K_2$ for some $t \geq 3$ containing $C$. 
Take a vertex $x_1$ of $H$ adjacent to $v_0$ 
other than $v_1$ and $v_3$. 
Then there exists a unique vertex $x_2$ which is adjacent to 
both vertices $x_1$ and $v_2$. 
Therefore, we can see that $v_0x_1x_2v_2$ is 
a $C$-avoiding $(v_0,v_2)$-path and so (a) does not hold. 
Hence the theorem holds.
\end{proof}

\subsection{A classification of the holes in a graph satisfying
the condition {\rm (E1)}}

\begin{Thm}\label{thm:strthm}
Let $G$ be a graph satisfying the condition {\rm (E1)},
and let $C$ be a hole in $G$.
Then exactly one of the following holds:
\begin{itemize}
\item[{\rm (A)}]
There is no $C$-avoiding path between two non-adjacent
vertices of $C$, and $X_C$ is a clique.
\item[{\rm (B)}]
The length of $C$ is equal to $4$,
and $C$ is contained
in an induced subgraph of $G$ isomorphic to
$K_t \Box K_2$ for some $t \geq 3$.
\item[{\rm (C)}]
The length of $C$ is equal to $4$,
and $C$ is contained in
an induced subgraph of $G$ isomorphic to
$K^m_2$ for some $m \geq 3$.
\end{itemize}
Moreover, if (B) happens, then $X_C$ is a clique, and if (C) happens, 
then there is no $C$-avoiding path between two non-adjacent vertices of $C$. 
\end{Thm}

\begin{proof}
First, we show by contradiction 
that (B) and (C) cannot happen at the same time. 
Suppose that both (B) and (C) hold. 
Let $C:=v_0 v_1v_2 v_3 v_0$ be a hole of length $4$ 
contained in both a prism $Y$ and an induced subgraph $K$ 
isomorphic to $K_{2,2,2}$. 
Without loss of generality, we may assume that $Y=[v_0 v_1 |x y|v_3v_2 ]$ 
and $V(K)=\{v_0, v_1, v_2, v_3, u_1, u_2\}$ 
for some $x$, $y$, $u_1$, $u_2 \in V(G)$. 
(See Figure \ref{fig-2-18}.) 
If $u_1 x \not\in E(G)$ and $u_1y \not\in E(G)$,
then the hole $u_1 v_0 x y v_2 u_1$
shares the two edges $u_1v_0$ and $u_1v_2$
with the hole $u_1v_0u_2v_2u_1$, which is a contradiction 
to the condition (E1).
If $u_1 x \not\in E(G)$ and $u_1y \in E(G)$,
then the hole $u_1 v_0 x y u_1$
shares the two edges $v_0 x$ and $xy$
with the hole $v_0 xy v_1 v_0$, which is a contradiction.
If $u_1 x \in E(G)$ and $u_1y \not\in E(G)$,
then the hole $u_1 x y v_2 u_1$
shares the two edges $xy$ and $y v_2$
with the hole $v_2 yx v_3v_2$, which is a contradiction.
Thus $u_1 x \in E(G)$ and $u_1y \in E(G)$.
By applying the same argument for $u_2$ instead of $u_1$,
we can show that $u_2 x \in E(G)$ and $u_2y \in E(G)$.
Then the hole
$u_1 v_0 u_2 y u_1$
shares the two edges $u_1v_0$ and $u_2v_0$
with the hole $u_1v_0u_2v_2u_1$, which is a contradiction. 
Thus we have shown that (B) and (C) cannot happen at the same time.

Now, we show the theorem. 
If the length of $C$ is at least $5$, then 
(A) holds by Lemmas \ref{lem:clique5} 
and \ref{newedgedisjoint} and neither (B) nor (C) can happen.
Therefore, we assume that the length of $C$ is equal to $4$. 
Suppose that (B) holds.
Then (A) does not hold 
since there is a $C$-avoiding path 
between two non-adjacent vertices of $C$, 
and (C) does not hold 
by the previous argument. 
Next suppose that (B) does not hold. 
Then it follows from Theorem \ref{thm:str} 
that there is no $C$-avoiding path 
between two non-adjacent vertices of $C$.  
If $X_C$ is a clique, then (A) holds 
and (C) cannot happen by Theorem \ref{thm:strXC}. 
If $X_C$ is not a clique, 
then (A) does not holds obviously 
and (C) happen by Theorem \ref{thm:strXC}. 
Hence exactly one of (A), (B), (C) holds. 

If (B) happens, then (C) cannot happen 
and so $X_C$ is a clique 
by Theorem~\ref{thm:strXC}. 
If (C) happens, then (B) cannot happen 
and so there is no $C$-avoiding path between 
two non-adjacent vertices of $C$ 
by Theorem \ref{thm:str}.
\end{proof}

\begin{figure}[h]
\psfrag{a}{$v_0$}
\psfrag{b}{$v_1$}
\psfrag{c}{$v_2$}
\psfrag{d}{$v_3$}
\psfrag{e}{$x$}
\psfrag{f}{$y$}
\psfrag{g}{$u_1$}
\psfrag{h}{$u_2$}
\begin{center}
\includegraphics[scale=0.7]{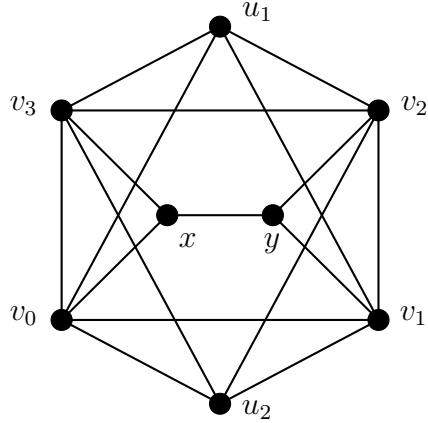}
\caption{A picture for Proof of Theorem \ref{thm:strthm}}
\label{fig-2-18}
\end{center}
\end{figure}

\section{Operations on graphs satisfying the condition {\rm (E1)}}
\subsection{Deleting an edge from a graph}

For a graph $G$ and an edge $uv$ in $G$, 
we denote by $G - uv$ 
the graph obtained from $G$ by deleting the edge $uv$.

\begin{figure}[h]
\psfrag{u}{$u$}
\psfrag{v}{$v$}
\psfrag{G}{$G$}
\psfrag{C}{$C$}
\psfrag{G'}{$G-uv$}
\begin{center}
\includegraphics[scale=0.7]{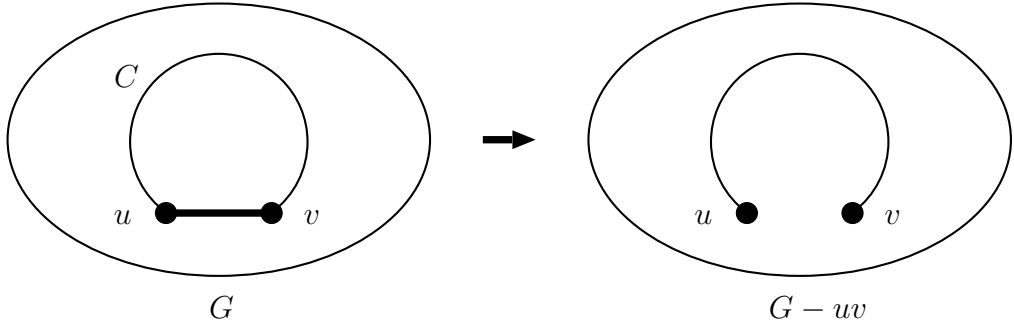}
\caption{Deleting an edge from a graph}
\label{fig-3-1}
\end{center}
\end{figure}

\begin{Lem}\label{noavoid}
Let $G$ be a graph satisfying the condition {\rm (E1)}.
Suppose that there exists a hole $C$ and
two adjacent vertices $u$ and $v$ on the hole $C$
such that
there is no $C$-avoiding $(u,v)$-path.
If $G$ is $K_{2,2,2}$-free,
then the following hold:
\begin{itemize}
\item[{\rm (1)}]
$G-uv$ also satisfies the condition {\rm (E1)}.
\item[{\rm (2)}]
If the number of holes of $G$ is $h$,
then that of $G-uv$ is at most $h-1$.
\item[{\rm (3)}]
$G-uv$ is also $K_{2,2,2}$-free.
\item[{\rm (4)}]
If $G$ is prism-free, then
$G- uv$ is still prism-free.
\end{itemize}
\end{Lem}

\begin{proof}
First we show that there is no new hole is created 
by deleting the edge $uv$ from $G$.
Suppose that
there is a hole $C'$ in $G- uv$
which is not a hole in $G$.
Then the edge $uv$ is a chord of $C'$ in $G$.
Now consider the two distinct $(u,v)$-sections $P_1$ and
$P_2$ of $C'$.
If $|E(P_1)| \ge 3$ or $|E(P_2)| \ge 3$, then
$P_1$ or $P_2$ is a $C$-avoiding $(u, v)$-path,
which contradicts the hypothesis.
Thus $|E(P_1)|=2$ and $|E(P_2)|=2$.
Then $P_1=uwv$ and $P_2=uw'v$ for some vertices $w$ and $w'$ of $G$.
Since $G$ does not have a $C$-avoiding $(u, v)$-path
by the hypothesis,
it holds that $\{w,w'\} \subseteq X_C \cup V(C)$.
However, if $w \in V(C)$,
then at least one of $uw$ or
$vw$ is a chord of $C$, which is a contradiction.
If $w, w' \in X_C$,
then $w$ and $w'$ are adjacent by Corollary~\ref{cor:clique4}
since $G$ is $K_{2,2,2}$-free.
Then the edge $ww'$ is a chord of $C'$ in $G-uv$,
which is a contradiction.
Therefore $G-uv$ has no hole other than the ones of $G$.
Since no new hole is created in $G-uv$, 
the graph $G-uv$ still satisfies 
the condition (E1).  
In addition, since $C$ is not a hole in $G-uv$ and the number of holes 
of $G-uv$ is at most $h-1$. 
Suppose that $G-uv$ has an induced subgraph $H$ isomorphic to $K_{2,2,2}$. 
Then $u$ and $v$ consist in a part of $H$. 
Let $x$ and $y$ be the vertices of another part of $H$. 
Obviously neither $x$ nor $v$ is on $C$ and $uxv$ and $uyv$ are paths in $G$. 
Since $x$ and $y$ are not adjacent, one of them should be 
a $C$-avoiding path by Corollary~\ref{cor:clique4}, which is a contradiction. 
A similar argument can applied to reach a contradiction 
if $G-uv$ contains a prism.
\end{proof}

\subsection{Adding an edge to a graph}

For a graph $G$ and two non-adjacent vertices $u$ and $v$ in $G$, 
we denote by $G+uv$ 
the graph obtained from $G$ by joining $u$ and $v$ by an edge. 

\begin{figure}[h]
\psfrag{u}{$u$}
\psfrag{v}{$v$}
\psfrag{G}{$G$}
\psfrag{K}{$K$}
\psfrag{G'}{$G+uv$}
\begin{center}
\includegraphics[scale=0.7]{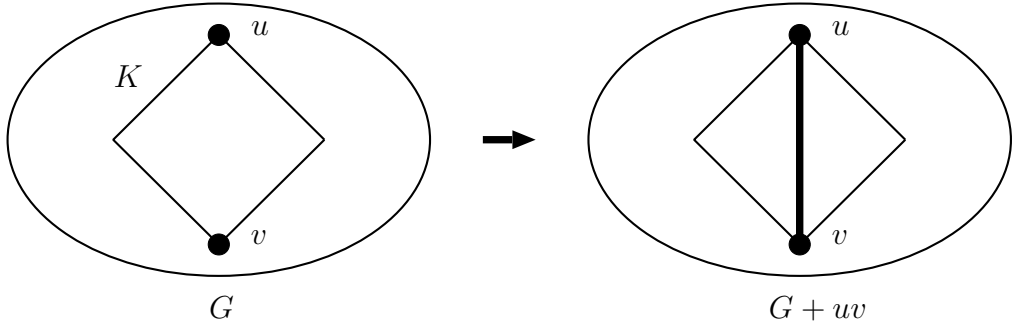}
\caption{Adding an edge to a graph}
\label{fig-3-2}
\end{center}
\end{figure}

\begin{Lem}\label{lem:addingdiagonaledge}
Let $G$ be a graph satisfying the condition {\rm (E1)} 
and $C:=v_0v_1v_2v_3v_0$ be a hole of length $4$ in $G$. 
Then adding an edge joining two non-adjacent vertices of $C$ 
does not create a new hole. 
\end{Lem}

\begin{proof}
By contradiction. Suppose a new hole $C'$ is created by adding 
an edge joining two vertices $v_0$ and $v_2$. 
Since $C'-v_0v_2$ together with edges $v_0v_1$ and $v_2v_1$ form 
a cycle of length at least $4$ in $G$ sharing two edges with $C$, 
$v_1$ must be adjacent to a vertex of $V(C')\setminus \{v_0,v_2\}$. 
Similarly, $v_3$ is adjacent to a vertex of $V(C')\setminus \{v_0,v_2\}$. 
If a vertex $x$ on $C$ is joined to both $v_1$ and $v_3$, 
then $v_0v_1xv_3v_0$ is a hole of $G$ sharing two edges $v_0v_1$ 
and $v_0v_3$ with the hole $C$, a contradiction. 
Let $y$ and $z$ be vertices of $C$ adjacent to $v_1$ and $v_3$, respectively, 
such that a shorter $(y,z)$-section $P$ of $C'$ is the shortest. 
Then no interior vertex of $P$ is adjacent to $v_0$ or $v_3$. 
Then $P$ together with the edges $v_1v_2, v_2v_3$, $v_1y$, $v_3z$ 
form a hole in $G$. However, this hole shares the two edges 
$v_1v_2$, $v_2v_3$ with $C$, which is a contradiction. 
Thus no new hole is created by adding an edge joining two vertices 
$v_0$ and $v_2$. By symmetry, no new hole is created by adding 
an edge joining two vertices $v_1$ and $v_3$. Hence adding an edge 
joining two non-adjacent vertices of $C$ does not create a new hole. 
\end{proof} 

\begin{Lem}\label{lem:K+e}
Let $G$ be a graph satisfying the condition {\rm (E1)}.
Let $m$ be the maximum integer such that $G$ contains an induced subgraph
$K$ isomorphic to $K_2^m$.
Let $u,v \in V(K)$ be two non-adjacent vertices of $K$.
Then the following hold:
\begin{itemize}
\item[{\rm (1)}]
$G+uv$ also satisfies the condition {\rm (E1)}.
\item[{\rm (2)}]
If $m \geq 3$ and the number of holes of $G$ is $h$,
then that of $G+uv$ is at most $h-2$.
\item[{\rm (3)}]
If $G$ is prism-free, then
$G+uv$ is still prism-free.
\end{itemize}
\end{Lem}

\begin{proof}
Since $u$ and $v$ are non-adjacent vertices of a hole of length $4$, 
by Lemma~\ref{lem:addingdiagonaledge}, $G+uv$ has no hole other than 
the ones of $G$ and so (1) holds. 

If $m \geq 3$, then $u$ and $v$ belong to
at least two distinct holes of length $4$ in $K$ and these holes in $G$ 
are not holes anymore in $G+uv$  as they
become $4$-cycles with chord $uv$. By the previous argument, 
no hole is created by joining $u$ and $v$ and so $G+uv$ 
has at most $h-2$ holes. Thus (2) holds. 

To show (3), suppose that $G+uv$ contains a prism $[xy|uv|zw]$. 
Then $x$ is not adjacent to $v$, and $y$ is not adjacent to $u$ in $G$. 
Therefore $x$ and $y$ cannot belong to $K$. 
Let $C^*$ be a hole of $K$ containing $u$ and $v$. 
Then $uxyv$ is a $C^*$-avoiding path, a contradiction. Thus (3) holds. 
\end{proof}

\subsection{Breaking prisms in a graph}

Suppose that a graph $G$ has
an induced subgraph $H$ isomorphic to
$K_t \Box K_2$ for $t \geq 3$.
Let $V(H)=V_x \cup V_y= \{ x_1, \ldots, x_t \}
\cup \{ y_1, \ldots, y_t \}$
where $V_x$ and $V_y$ are isomorphic to $K_t$
and $x_iy_i \in E(H)$.
Define a graph $G_{/H/}$ by
$V(G_{/H/})=V(G)$ and $E(G_{/H/})=E(G)
\cup \{x_1y_2, x_2y_3, \ldots, x_{t-1}y_t, x_ty_1 \}$.

\begin{figure}[h]
\psfrag{K}{$K_t$}
\psfrag{G}{$G$}
\psfrag{H}{$H$}
\psfrag{G'}{$G_{/H/}$}
\begin{center}
\includegraphics[scale=0.7]{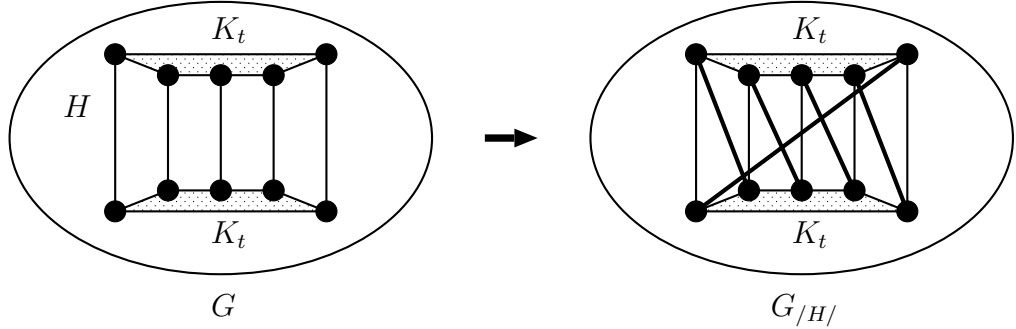}
\caption{Breaking prisms in a graph}
\label{fig-3-3}
\end{center}
\end{figure}

\begin{Lem}\label{bkprism}
Let $G$ be a graph satisfying the condition {\rm (E1)}.
Suppose that $G$ has
an induced subgraph $H$ isomorphic to
$K_t \Box K_2$ for $t \geq 3$.
Then the following hold:
\begin{itemize}
\item[{\rm (1)}]
$G_{/H/}$ also satisfies the condition {\rm (E1)}.
\item[{\rm (2)}]
If the number of holes of $G$ is $h$,
then that of $G_{/H/}$ is at most $h-t$.
\end{itemize}
\end{Lem}

\begin{proof} 
By Lemma~\ref{lem:addingdiagonaledge}, $G_{/H/}$ contains 
no new hole other than the ones in $G$. 
Moreover, adding an edge between two non-adjacent vertices 
of a hole of length $4$ in $H$ 
breaks the hole. Thus at least $t$ holes of $G$ are broken 
in $G_{/H/}$. Hence (1) and (2) hold. 
\end{proof}

\section{Proof of Theorem 1.7}   
\subsection{Outline of the proof}

Let $G$ be a graph satisfying the condition {\rm (E1)}.
Then exactly one of the following three cases happens: \\
{\rm (Case A)}:
$G$ is prism-free and $K_{2,2,2}$-free. \\
{\rm (Case B)}:
$G$ is prism-free and 
$G$ has an induced subgraph $K$ isomorphic to $K_{2,2,2}$. \\
{\rm (Case C)}:
$G$ has an induced subgraph $Y$ isomorphic to a prism. \\

Suppose that $G$ has exactly $h$ holes $C_1$, $C_2$, \ldots, $C_h$.
For each $t \in [h]:=\{1, \ldots, h \}$, we let
\[
C_t= v_{t,0} v_{t,1} \ldots v_{t,m_t -1} v_{t,0},
\]
where $m_t$ is the length of the hole $C_t$.
For $t \in [h]$ and $i \in \{0, \ldots, m_t-1\}$,
let
$S_{t,i}$ be the set of vertices each of which is 
an internal vertex of a $C_t$-avoiding walk from $v_{t,i}$ to $v_{t,i+1}$, i.e., 
\begin{equation}
S_{t,i} := \bigcup_{W \in \cW_{t,i}} V(W)  
\setminus \{ v_{t,i}, v_{t,i+1} \},  
\end{equation}
where $\cW_{t,i}$ denotes the set of all $C_t$-avoiding $(v_{t,i},v_{t,i+1})$-walks 
in $G$. 

We will prove Theorem \ref{thm:main} by induction on the number of holes 
by taking the following steps: \\
{\bf Step 1:}  
We prove (Case A) by using
the operation ``Deleting an edge from a graph". \\
{\bf Step 2:} 
(Case B) is reduced to (Case A) by
using the operation ``Adding an edge to a graph". \\
{\bf Step 3:} 
(Case C) is reduced to (Case B) by using the operation
``Breaking prisms in a graph".

\subsection{Proof for (Case A)}

Consider (Case A).
If there are no holes of length $4$ in $G$,
then all the holes are independent.
Therefore it holds that $k(G) \leq h(G) + 1$
by Theorem \ref{thm:indep}.
Suppose that there is a hole of length $4$ in $G$.
Since $G$ is prims-free and $K_{2,2,2}$-free, 
the condition (a) in Theorem \ref{thm:strXC} and
the condition (a) in Theorem \ref{thm:str}
hold for each of the holes of length $4$ in $G$.


\begin{Thm} \label{thm:mainCaseA}
Let $G$ be a prism-free $K_{2,2,2}$-free graph satisfying the condition
{\rm (E1)} with exactly $h$ holes and let $Q$ be a clique of $G$ containing 
an edge of a hole.
Then there exists an acyclic digraph $D$ such that
$C(D)=G \cup I_{h+1}$ and the vertices of $Q$ have no in-neighbors in $D$. 
Consequently, the competition number of $G$ is at most $h+1$.
\end{Thm}

\begin{proof}
We shall prove the theorem by induction on $h$.
The competition number of $G$ is at most $1$ if $h = 0$ 
by Theorem~\ref{roberts}. Since there is no hole, 
there is no clique containing an edge of a hole and so the theorem is true 
if $h=0$.

We assume that $h \geq 1$ and the theorem is true for any prism-free 
$K_{2,2,2}$-free graph satisfying the condition (E1) with less than $h$ holes. 
Suppose that $(\{v_{t,i}, v_{t,i+1}\} \cup S_{t,i}) \cap Q \neq \emptyset$
for some $t \in [h]$ and $i \in \{0, \ldots, m_{t}-1\}$.
Take $x \in (\{v_{t,i}, v_{t,i+1}\} \cup S_{t,i}) \cap Q$. 
If $x=v_{t,i}$, then $v_{t,j} \not\in Q$ for any $j \neq i-1, i+1$. 
Since $m_t \ge 4$, $i+3 \not\equiv i \pmod{m_t}$.
Suppose that $x' \in (\{v_{t,i+2}, v_{t,i+3}\} \cup S_{t,i+2}) \cap Q $. 
Then $x' \in S_{t,i+2} \cap Q $. Since $Q$ is a clique, 
$x'$ is adjacent to $x$. 
By the definitions of $S_{t,i}$ and $S_{t,i+2}$, 
there exists a $C_t$-avoiding $(v_{t,i},{v_{t,i+2}})$-walk, 
which contradicts Theorem~\ref{thm:strthm}. 
Thus $(\{v_{t,i+2}, v_{t,i+3}\} \cup S_{t,i+2}) \cap Q =\emptyset$. 
If $x=v_{t,i+1}$, we can show that 
$(\{v_{t,i-2}, v_{t,i-1}\} \cup S_{t,i-2}) \cap Q =\emptyset$ 
by a similar argument. 
If $x \in S_{t,i} \cap Q$ and 
$x' \in (\{v_{t,i-2}, v_{t,i-1}\} \cup S_{t,i-2}) \cap Q$, 
then $x$ and $x'$ are adjacent and there exists 
a $C_t$-avoiding $(v_{t,i-2},v_i)$-walk. 
Now Theorem~\ref{thm:strthm} is violated. 
Thus $ (\{v_{t,i-2}, v_{t,i-1}\} \cup S_{t,i-2}) \cap Q=\emptyset$. 
We have just shown that  there exists $j \in \{0, \ldots, m_t-1\}$ 
such that 
\[
(\{v_{t,j}, v_{t,j+1}\} \cup S_{t,j}) \cap Q = \emptyset.
\] 
We claim that no vertex of $S_{t,j}$ is adjacent to a vertex of 
$V(G)- (X_t \cup V(C_t) \cup S_{t,j})$. 
Suppose otherwise, there is a $C_t$-avoiding $(v_{t,j},v_{t,j+1})$-walk 
$W$ that contains an internal vertex $x$ adjacent 
to a vertex $y \in V(G)- (X_t \cup V(C_t) \cup S_{t,j})$.
The walk $W'$ obtained  by
replacing the term $x$ with $xyx$ in the sequence of $W$ 
is a $C_t$-avoiding $(v_{t,j}, v_{t,j+1})$-walk, 
which contradicts the assumption that $y \not\in S_{t,j}$. 
On the other hand, there is no $C_t$-avoiding path connecting 
two non-adjacent vertices of $C_t$ by Theorem~\ref{thm:strthm}, 
and so no vertex of $S_{t,j}$ is adjacent to a 
vertex of $V(C_t)- \{v_{t,j}, v_{t,j+1}\}$. 
Hence, $X_t \cup \{v_{t,j}, v_{t,j+1}\}$ 
is a vertex cut of $G$ and no vertex in $S_{t,j}$ 
belongs to the component that contains 
$V(C_t) - \{v_{t,j}, v_{t,j+1}\}$. 

Consider 
the subgraph $G_1$ of the graph $G$
induced by $V(G) - S_{t,j}$ and
the subgraph $G_2$ of $G$ induced by
$X_t \cup \{v_{t,j}, v_{t, j+1}\} \cup S_{t,j}$.
Since $V(G_1) \cap V(G_2) = X_t \cup \{v_{t,j}, v_{t,j+1}\}$
is a vertex cut of $G$ which is a clique, the vertex set of a hole is 
contained in either $V(G_1) \setminus V(G_2)$ or 
$V(G_2) \setminus V(G_1)$. Thus, if $h_1$ is the number of holes of $G_1$, 
then $h_2 := h - h_1$ 
is the number of holes of $G_2$.
It is obvious that $E(G_1) \cup E(G_2)=E(G)$ and that both $G_1$ 
and $G_2$ are prism-free, $K_{2,2,2}$-free, and satisfy the condition (E1). 

Since the hole $C_t$ is not in $G_2$, we have $h_2 < h$.
By the induction hypothesis, there exists an acyclic digraph $D_2$ 
such that $C(D_2) = G_2 \cup I_{h_2 +1}$
and the vertices of $X_t \cup \{v_{t,j}, v_{t,j+1} \}$
have only outgoing arcs in $D_2$.
Notice that $C_t$ is a hole in $G_1$ which has no
$C_t$-avoiding walk from $v_{t,j}$ to $v_{t,j+1}$.
By Lemma \ref{noavoid}, $G_1 - v_{t,j} v_{t,j+1}$
has exactly $h_1 -1$ holes and satisfy the condition (E1).
By the choice of $j$, $Q$ is a clique in
$G_1 - v_{t,j} v_{t,j+1}$ and, by the induction hypothesis,
there exists an acyclic digraph $D_1$ such that
$C(D_1)=(G_1 - v_{t,j} v_{t,j+1}) \cup I_{h_1}$
and the vertices of $Q$ have only outgoing arcs in $D_1$.
We now define a digraph $D$ by $V(D) = V(D_1) \cup V(D_2)$
and $A(D) = A(D_1) \cup A(D_2)$.
It is easy to check that $D$ is an acyclic digraph
with the vertices of $Q$ having only outgoing
arcs in $D$. 
Since $V(G_1) \cap V(G_2)= X_t \cup \{v_{t,j},v_{t,j+1}\}$ 
the vertices of which have only outgoing arcs in $D_2$, no new edge 
is added to $C(D)$ other than the ones in $G_1$ or $G_2$. 
Thus, $E(C(D)) = (E(G_1)\setminus \{v_{t,j}v_{t,j+1}\}) 
\cup E(G_2)$ $=E(G_1) \cup E(G_2)= E(G)$. 
Hence $C(D)=G \cup I_h$. 
Consequently, $k(G) \leq h + 1$. 
\end{proof}

\subsection{Reducing (Case B) to (Case A)}

\begin{Thm}\label{thm:mainCaseB}
Let $G$ be a graph satisfying the condition {\rm (E1)}
with exactly $h$ holes.
If $G$ is prism-free, then the competition number of
$G$ is at most $h+1$.
\end{Thm}

\begin{proof}
By induction on $h$. 
If $h=0$, then the theorem follows from 
Theorem~\ref{roberts}. 
Assume that the theorem is true 
for any prism-free graph satisfying the condition (E1) 
with less than $h$ holes. 
Let $m$ be the maximum integer such that $G$ 
contains an induced subgraph $K$ isomorphic to $K_2^m$. 
If $m \leq 2$, then $G$ is $K_{2,2,2}$-free. 
By Theorem~\ref{thm:mainCaseA}, the theorem holds. 
Suppose that $m \ge 3$. 
Let $u$ and $v$ be two non-adjacent vertices of $K$.
By Lemma \ref{lem:K+e}, 
the graph $G':=G+uv$ 
is a prism-free graph satisfying the condition {\rm (E1)}
and has at most $h-2$ holes.
Therefore, by induction hypothesis,
there exists an acyclic digraph $D'$ such that
$C(D')=G' \cup I_{h-1}$.

In the following, we shall construct an acyclic digraph $D$
such that $C(D)=G \cup I_{h+1}$ from $D'$. 
We first look at the vertices in 
$N_{D'}^+(u) \cap N_{D'}^+(v)$ 
which play as prey of cliques containing the edge $uv$ in $G'$. 
Let $N_{D'}^+(u) \cap N_{D'}^+(v)=\{w_1, \ldots, w_p\}$
for some integer $p \ge 1$. Let $H_i$ be the subgraph of $G$ 
induced by $N_{D'}^-(w_i)$.
In $G$, the edges of $H_i$ are covered by exactly two cliques
$N_{D'}^-(w_i) \setminus \{u\}$ and $N_{D'}^-(w_i) \setminus \{v\}$ 
unless $N_{D'}^-(w_i)=\{u,v\}$. 
Furthermore, 
since $w_i$ is a common out-neighbor of $u$ and $v$,
\[
\bigcup_{i=1}^p N_{D'}^-(w_i)
\subseteq (N_{G}[u] \cap N_G[v])
\subseteq X_K \cup V(K),
\]
where 
$X_K$ is defined by (\ref{eq:XK})
and the last inclusion follows from Lemma~\ref{multi} (2).
Thus
\[
N_G[v] \cap \bigcup_{i=1}^p N_{D'}^-(w_i)
\subseteq N_G[v] \cap (X_K \cup V(K)).
\]
The vertices in  $N_G[v] \cap (X_K \cup V(K))$ are
covered by exactly two cliques in $G$.
We denote those cliques by $Z_1$ and $Z_2$.
We define a digraph $D$ as follows:
\begin{eqnarray*}
V(D) &=& V(D') \cup \{ z_1, z_2 \}; \\
A(D) &=& 
\left(
A(D') \setminus \bigcup_{i=1}^p
\{(v,w_i)\} 
\right)
\cup \{(x,z_1) \mid x \in Z_1\}
\cup \{(x,z_2) \mid x \in Z_2\}. 
\end{eqnarray*}
Then it is obvious that $D$ is acyclic 
and $E(C(D)) \subset E(G)$. 
By removing the arcs in 
$\bigcup_{i=1}^p \{(v,w_i)\}$ from $D'$, 
the competition graph of the new digraph loses the edges joining $v$ 
and the vertices in $\bigcup_{i=1}^p N_{D'}^-(w_i)$. 
Those edges are contained in the subgraph induced by 
$N_G[v] \cap (X_K \cup V(K))$ as we argued above. 
Thus those edges are contained in the cliques formed by $Z_1$ or $Z_2$. 
Hence $C(D)=G \cup I_{h+1}$ and so $k(G) \le h+1$. 
\end{proof}

\subsection{Reducing (Case C) to (Case B)}

Now we complete the proof of Theorem \ref{thm:main}.

\begin{proof}[Proof of Theorem \ref{thm:main}]
By induction on $h$.  
If $h=0$, then the theorem follows from Theorem~\ref{roberts}. 
Assume that the theorem is true for any graph satisfying the condition (E1) 
with less than $h$ holes. 
Let $t$ be the maximum integer such that $G$ 
contains an induced subgraph $H$ isomorphic to $K_t \Box K_2$. 
If $t \leq 2$, then $G$ is prism-free. 
By Theorem~\ref{thm:mainCaseB}, the theorem holds. 
Suppose that $t \geq 3$. 
Consider the graph $G_{/H/}$.
Then, by Lemma~\ref{bkprism}, $G_{/H/}$ satisfies the condition (E1) 
and the number of holes in $G_{/H/}$ is $h-t$ which is less than $h$.
By the induction hypothesis, there exists an acyclic digraph $D'$ 
such that $C(D')=G_{/H/} \cup I_{h-t+1}$. Take $i \in \{1, \ldots, t\}$ 
and  $w \in N_{D'}^+(x_i) \cap N_{D'}^+(y_{i+1})$. 
Then $N_{D'}^-(w) \subset  X_{C_i} \cup \{x_i, x_{i+1},y_{i+1}\}$ 
or $N_{D'}^-(w) \subset  X_{C_i} \cup \{x_i, y_i, y_{i+1}\}$ 
where $C_i:=x_iy_iy_{i+1}x_{i+1}x_i$ 
(identify $x_{t+1}$ and $y_{t+1}$ with $x_1$ and $y_1$, respectively). 
By Theorem~\ref{thm:strthm}, $X_{C_i}$ is a clique in $G$ 
and so $N_{D'}^-(w) \setminus \{x_i\}$ is a clique in $G$. 
Now we define a digraph $D$ by 
\begin{eqnarray*}
V(D) &=& V(D') \cup \{z_1, \ldots, z_t \}, \\
A(D) &=& 
\left(
A(D') \setminus \bigcup_{i=1}^{t}
\{(x_i,w) \mid w \in N^+_{D'}(x_i) \cap N^+_{D'}(y_{i+1})\}
\right) \\
&& 
\cup \bigcup_{i=1}^{t} \{(x, z_i) \mid x \in \{x_i,x_{i+1} \} \cup X_{C_i}\}. 
\end{eqnarray*}
Obviously $D$ is acyclic. Note that 
\[
N_D^-(w) = \left\{
\begin{array}{ll} 
N_{D'}^-(w) \setminus \{x_i\}  
& \mbox{ if } w \in N^+_{D'}(x_i) \cap N^+_{D'}(y_{i+1}) 
\mbox{ for some } i \in \{1, \ldots, t\}; \\
N_{D'}^-(w) & \mbox{ otherwise.}\end{array}\right.\]
Also notice that deleting the arcs in $\bigcup_{i=1}^{t}
\{(x_i,w) \mid
w \in N^+_{D'}(x_i) \cap N^+_{D'}(y_{i+1})\}$ from $D'$ 
may remove edges only in the clique $\{x_i,x_{i+1} \} \cup X_{C_i}$ 
for some $i \in \{1, \ldots, t\}$ from $C(D')$. 
From these observations, we can conclude that $C(D)=G \cup I_{h+1}$. 
Hence $k(G) \leq h+1$.
\end{proof}

\begin{Cor}
Let $G$ be a graph with exactly $h$ holes satisfying the following property:
\begin{itemize}
\item{}
For any two distinct holes $C$ and $C'$, $|V(C) \cap V(C')| \leq 2$.
\end{itemize}
Then the competition number of $G$ is at most $h+1$.
\end{Cor}

\begin{proof}
It follows from Theorem \ref{thm:main} and Proposition \ref{prop:E1V2}. 
\end{proof}



\begin{thebibliography}{99}

\bibitem{ck}
H. H. Cho and S. -R. Kim:
The competition number of a graph having exactly one hole,
{\it Discrete Mathematics} {\bf 303} (2005) 32--41.

\bibitem{co}
{J. E. Cohen}:
{Interval graphs and food webs: a finding and a problem},
{\it Document 17696-PR}, RAND Corporation,
Santa Monica, CA (1968).

\bibitem{Kamibeppu}
A. Kamibeppu:
An upper bound for the competition numbers of graphs,
{\it Discrete Applied Mathematics} {\bf 158} (2010) 154--157.

\bibitem{kimsu}
{S. -R. Kim}:
{The competition number and its variants},
in J. Gimbel, J.W. Kennedy, and L.V. Quintas (eds.),
{\it Quo Vadis Graph Theory?},
{\it Annals of Discrete Mathematics}, Vol. \textbf{55}
(1993) 313--325.

\bibitem{compone}
S. -R. Kim:
Graphs with one hole and competition number one,
{\it Journal of the Korean Mathematical Society} {\bf 42} (2005) 1251--1264.

\bibitem{kr}
S. -R. Kim and F. S. Roberts:
Competition numbers of graphs with a small number of triangles,
{\it Discrete Applied Mathematics} {\bf 78} (1997) 153--162.

\bibitem{LKKS}
S. -R. Kim, J. Y. Lee, and Y. Sano:
The competition number of a graph whose holes do not overlap much, 
{\it Discrete Applied Mathematics} {\bf 158} (2010) 1456--1460.

\bibitem{twoholes}
J. Y. Lee, S. -R. Kim, S. -J. Kim, and Y. Sano:
The competition number of a graph with exactly two holes,
{\it Ars Combinatoria} {\bf 95} (2010) 45--54.

\bibitem{maxclique}
J. Y. Lee, S. -R. Kim, S. -J. Kim, and Y. Sano:
Graphs having many holes but with small competition numbers,
{\it Preprint}, arXiv:0909.5311.

\bibitem{indep}
B. -J. Li and G. J. Chang:
The competition number of a graph with exactly $h$ holes, all
of which are independent,
{\it Discrete Applied Mathematics} {\bf 157} (2009) 1337--1341.

\bibitem{twoholes2}
B. -J. Li and G. J. Chang: 
The competition number of a graph with exactly two holes,
{\it Journal of Combinatorial Optimization}, 
DOI: 10.1007/s10878-010-9331-9

\bibitem{lu}
{J. R. Lundgren}:
{Food webs, competition graphs,
competition-common enemy graphs, and niche graphs, in
Applications of Combinatorics and Graph Theory to the Biological
and Social Sciences},
{\it IMH Volumes in Mathematics and Its Application} {\bf  17}
Springer-Verlag, New York, (1989) 221--243.

\bibitem{op}
{R. J. Opsut}:
{On the computation of the competition number of a graph},
{\it SIAM Journal on Algebraic and Discrete Methods} {\bf 3} (1982) 420--428.

\bibitem{RayRob}
{A. Raychaudhuri and F. S. Roberts}: 
{Generalized competition graphs and their applications}, 
in P. Br\"{u}cker and
R. Pauly (eds.), Methods of Operations Research, Vol.\ 49,
Anton Hain, K\"{o}nigstein, West Germany, (1985) 295--311.

\bibitem{cn}
{F. S. Roberts}:
{Food webs, competition graphs, and the boxicity of ecological phase space},
{\it Theory and applications of graphs (Proc. Internat. Conf.,
Western Mich. Univ., Kalamazoo, Mich., 1976)} (1978) 477--490.

\end{thebibliography}
\end{document}